\documentclass{amsart}
\usepackage{amsfonts}
\usepackage[active]{srcltx}
\usepackage{amsmath}
\usepackage{amssymb}
\usepackage{amsthm}
\usepackage{mathtools}
\usepackage{shuffle}
\usepackage{bm} %for bold math symbol
\usepackage{wrapfig}

\usepackage{enumitem}
\usepackage{wasysym}
\usepackage{mathscinet}
\usepackage{verbatim}
\usepackage{graphicx}
\usepackage[
bookmarks=true,         %generates bookmarks for all entries in the "Table of Contents"
bookmarksnumbered=true, %includes chapter-/header-/section-/subsection-/... -
colorlinks=true, pdfstartview=FitV, linkcolor=blue, citecolor=blue,
urlcolor=blue]{hyperref}
\usepackage{microtype}
\usepackage{comment}
\usepackage{xcolor,cancel}
\usepackage[normalem]{ulem}
\usepackage{amsmath}
\newcommand{\stkout}[1]{\ifmmode\text{\sout{\ensuremath{#1}}}\else\sout{#1}\fi}
\usepackage{xcolor,cancel}

\theoremstyle{plain}
\newtheorem{theorem}{Theorem}[section]

\newtheorem{lemma}[theorem]{Lemma}
\newtheorem{problem}[theorem]{Problem}
\newtheorem{proposition}[theorem]{Proposition}

\renewenvironment{proof}[1][Proof]{\textbf{#1.} }{\ \rule{0.5em}{0.5em} \par }
\theoremstyle{remark}

\newtheorem{assumption}[theorem]{Assumption}

\theoremstyle{definition}
\newtheorem{remark}[theorem]{Remark}

\def\wt{\widetilde}

\def\Om{{\Omega}}
\def\om{{\omega}}

\def\diag{{\rm diag}}

\def\om{\omega}
\def\Om{\Omega}

\def\wt{\widetilde}

\def\wt{\widetilde}

\DeclareMathOperator*{\Div}{\mathrm{div}}

\def\om{\omega}
\def\Om{\Omega}

\def\mcH{\mathcal{H}}

\let\Section=\section
\def\section{\setcounter{equation}{0}\Section}

\title[Null controllability of degenerated     parabolic equations]{Null controllability of   n-dimensional  parabolic equations  degenerated on partial  boundary }
\date{April 2024}
\author[W. Wu, Y.Hu, H. Sun, D. Yang]{Weijia Wu$^{1}$, Yaozhong Hu$^{2}$, Hongli Sun$^{3,*}$,   Donghui Yang$^{1}$ }
\thanks{${}^*$Corresponding author: honglisun@126.com}
\address{$^1$ School of Mathematics and Statistics, Central South University, Changsha 410083, China} 
\address{$^2$  Department of Mathematical and Statistical Sciences, University of Alberta, Edmonton, AB T6G 2G1, Canada} 
\address{$^3$School of Mathematics, Physics and Big data, Chongqing University of Science and Technology,
	Chongging 401331, China}
\email{weijiawu@yeah.net}
\email{yaozhong@ualberta.ca}
\email{honglisun@126.com}
\email{donghyang@outlook.com}

\begin{document}
	\begin{abstract} %We consider a class of two-dimensional degenerate parabolic equations with general symmetric coefficients. We provide results on improving the regularity and establish Carleman estimates for the corresponding equations by constructing specialized weight functions. As a result, we prove the null controllability of the associated equations. Furthermore, we present a specific example to illustrate the effectiveness of our methodology.
		 This paper   extends 
		 the  Carleman estimates  to high  dimensional   parabolic equations with highly degenerate   symmetric coefficients on a bounded   domain of Lipschitz   boundary and use these estimates  to 
		   study the  controllability the corresponding equations.  Due to the nonsmoothness and   degeneracy of   boundary, the partial  integration by parts in Carleman
		   estimates have no meaning on the degenerate 
		   and nonsmooth parts of the boundary. To get around of this difficulty, we    construct  special weight function, and   transform some integral terms in degenerate regions into a non-degenerate ones carefully so that the obtained  Carleman estimates can still be used to the controllability problem. Our results includes some well-known works   as some special cases
		   as well as some interesting new examples.      
	\end{abstract}
	\pagestyle{myheadings}
	\thispagestyle{plain}
	\markboth{DEGENERATE PARABOLIC EQUATIONS WITH GENERAL SYMMETRIC COEFFICIENTS}{}
	
	\maketitle
	%	\tableofcontents
	
	\section{Introduction}
	%	Controllability is a fundamental concept in control theory. It holds great importance in solving control problems within linear systems. The study of controllability for parabolic equations has a rich history spanning half a century, (see \cite{fattorini1971exact,fattorini1974uniform,russell1973unified,lebeau1995controle,fursikov1996controllability,emanuilov1995controllability}), and can be categorized into two main branches: the controllability of non-degenerate parabolic equations and the controllability of degenerate parabolic equations. While there has been significant progress in analyzing the controllability of non-degenerate parabolic equations across various fields, research on the controllability of degenerate parabolic equations still remains relatively limited. 
	Degenerate equations have generated a great amount of   interest   in recent years (e.g.  \cite{Allal2021null, allal2020lipschitz,benoit2023null,boutaayamou2018carleman,boutaayamou2016null,de2023null, liu2019carleman, liu2024observability, wu2020carleman}) since  they can  used to describe many physical phenomena. For example, the famous Crocco equation is a degenerate parabolic equation,  reflecting  the compatibility relationship between the change in total energy and entropy in steady flow and vorticity. Tornadoes follow the Crocco equation during the rotation process, and thus, the study of controllability and optimal control problems of the Crocco equation is of great significance in meteorology (see \cite{martinez2003regional}). 
	Similarly, the famous Black-Scholes equation (see \cite{sakthivel2008exact}), which is widely studied in finance, and the Kolmogorov equation (see \cite{calin2009heat,calin2010heat}) are also degenerate parabolic equations with critical practical applications in real life.  
	%	Therefore, the study of the control problems of degenerate parabolic equations is of significant importance.
	
	In recent years, the controllability of  degenerate parabolic equations has received great attention. In \cite{cannarsa2004persistent}, the authors introduce  the concepts of regional null controllability and regional persistent null controllability, and obtain   the regional controllability for a   linearized Crocco type equation and also for the nondegenerate heat equation in unbounded domains. This stimulates   further studies on  the controllability of one-dimensional degenerate heat equations and other one-dimensional degenerate parabolic equations   in subsequent works \cite{Allal2021null,araujo2022boundary, CA5,flores2010carleman,flores2020null, lin2007some,wang2023carleman}.  For  the controllability of high-dimensional degenerate equations,  the two-dimensional Grushin-type operator has been  studied  (e.g.  \cite{anh2013null,banerjee2022carleman,beauchard20152d}) 
	because this 
	%	 Grushin 
	operator %is a   class of   elliptic operators
	%, particularly in higher dimensions where it 
	typically degenerates along only one edge, which  often renders its study more tractable than other more intricately structured operators.
	The idea  (e.g. \cite{beauchard20152d}) to deal with  this  Grushin-type operator    is to  establish  the Carleman estimate   by using the Fourier decomposition,
	and then    the null controllability is obtained.  %Besides, 
	There are also some   results on the controllability of other two-dimensional degenerate equations,   
	%	Among the references, let us mention one  that is closely relevant to this paper:  the authors \cite{cannarsa2016global}  use  Carleman estimates to study more general than Grushin-type operator for degenerate parabolic operators in dimensions higher than one. Compared to the one-dimensional theory, \cite{cannarsa2016global} provides more technical results,  the authors present an effective approach to obtain the global Carleman estimates of two-dimensional degenerate equations.
	%	In \cite{cannarsa2016global}, the author considered such a system:
	one of which is 
	\begin{equation}\label{2016}
		\begin{cases}
			\partial_{t}u - \Div(A\nabla u) + bu=\chi_{\omega_{0}}g, & \mbox{in} \ Q,  \\
			u=0 \ \mbox{if} \ \alpha\in (0,1),  & \mbox{on} \ \Sigma,  \\
			A\nabla u\cdot \nu =0 \ \mbox{if} \ \alpha\in (1,2),  & \mbox{on} \ \Sigma,  \\
			u(0)= u_{0},  & \mbox{in} \ \Omega,  
		\end{cases}
	\end{equation}
	where   $\Omega$ is a bounded open set of $\mathbb{R}^2$ with boundary $\Gamma$,  $Q:= \Omega\times(0,T)$
	and $A=(a_{ij})_{i,j=1,2}$ is a symmetric and nonnegative definite $2\times 2$ matrix valued function defined on $\overline{\Omega}$.  Under the the conditions that the boundary $\Gamma$ is $C^4$ and one of the eigenvalue of $A$ is not degenerated, the authors of \cite{cannarsa2016global}  can still  obtain   global Carleman estimates    and then further obtain the controllability for above    equation. 
	This paper is motivated by the above work and our contributions are summarized as follows. 
	\begin{itemize}
		\item [($1$)] We   allow the   boundary $\Gamma$ to be Lipschitzian. As can be seen in     \cite[Chapter 2]{cannarsa2016global}, the $C^4$  condition on the boundary is a necessary  to construct a $C^3$- diffeomorphism   so  that it is possible to establish the Carleman estimate near the boundary. 
		It is possible to improve from  $C^4$ to $C^{3\frac{1}{2}}$. But it is hard to go further in their argument to allow the Lipschitz boundary. We need to use a new method, which will be explained later in details. 
		\item [($2$)] We  shall consider  general   $n$-dimensional degenerate parabolic equations 
		instead of only two-dimensional,  which is more challenging. To our best knowledge, there has been no previous   work  addressing  the controllability  problem  in higher dimensions (e.g.  $n\ge 3$). This paper is motivated to fill this gap. 
		%			Our idea is to  obtain  a Carleman estimate to establish the null controllability of those equations, where the degenerate coefficient is  a general symmetric function, and the control region is located near the boundary. 
		\item [($3$)]  We allow  the  boundary to be  degenerate partially, which includes  whole boundary degeneracy for  system \eqref{2016} studied in \cite{cannarsa2016global}   as  a special case.
		\item [($4$)] Our control system \eqref{1.1} can degenerate on  all eigenvalues of $A(x)$, and we make very limited assumptions     (see Assumption \ref{assume1}). This significantly extends the applicability of our results. 
	\end{itemize}

	Compared with non degenerate equations, the solution of degenerate ones  does not have sufficient regularity in the whole domain $\Omega$, so whether the partial  integration by parts  holds and whether the boundary integration term is meaningful are  key issues we need to address.  The way most researchers adopt is   to   establish   the weighted Hardy inequality to prove the  existence of traces so that to make the  boundary integral meaningful.  We refer to 
	%	   So far, this method has been used in a few papers studying the Carleman estimates of two-dimensional degenerate parabolic equations
	(\cite{BSV,cannarsa2016global}) for using  this method to 
	obtain  the Carleman estimates of two-dimensional degenerate parabolic equations. Although  requiring   more assumptions, it  makes the selection of control regions more flexible in the obtained Carleman estimates.  For example, in \cite{cannarsa2016global}, the control  region $\omega_0$ can be very general and   needs only to be a nonempty open subset of $\Omega$. In \cite{BSV}, the authors weakens the condition for boundary $\Gamma$, but they require   $A(x)$ to be a specific functional form: $A(x) = \diag(x_1^{\alpha_1},x_2^{\alpha_2})$.  Moreover,    in  \cite{BSV} the control region $\omega_0$ cannot be as general in   \cite{cannarsa2016global} and it   needs to contain the  partial degenerate boundary. 
	As in    \cite{BSV} the trade-off  of  our method is that  the control region $\omega_0$ cannot be general either and  needs to include the entire degenerate boundary as well as the  non $C^2$  boundary.  In a future work,  we plan to use the approximation idea to prove the approximate controllability when the control region is an open subset of $\Om$, and the method is similar to our other work \cite{liu2023some}, where  we use the $\epsilon$ approximation theory to control the integrals over degenerate regions by those over non-degenerate regions to study controllability in the case of interior single-point degeneracy.  
	It is worth noting that there are significant differences between this article and \cite{liu2023some}, where the Carleman estimates are for high-dimensional but with single interior point degeneracy. There is a fundamental difference between internal degeneracy and boundary degeneracy, with the latter being more affected by boundary conditions. In \cite{liu2023some}, authors have established Carleman estimates for both cases: $0 \in \omega_0$ and $0 \notin \omega_0$, and the focus is on the carleman estimates in case $0 \notin \omega_0$. For each solution of the dual equation, it is introduced a sufficiently small $\epsilon$,  such that the integral terms over the degenerate region can be estimated by the integral terms over the non-degenerate region.
	This makes the coefficients $C$ in the final Carleman estimate dependent on the solution of dual equation, which implies only approximate controllability but not null controllability.
	
	%	Overall, it is easy to see the differences and advantages between this paper and other papers on Carleman estimates of degenerate equations. 
	Now let us give more details about the new method we mention earlier in (1).  Due to our weaker  smoothness conditions on the  boundary and more general assumptions on  $A(x)$,  we cannot expect to obtain the existence of traces on the degenerate boundary. Therefore, we cannot expect to  use the Hardy inequality or to prove the existence of traces on boundary $\Gamma$. Thus,  we need to follow   a   new method to get around of  the problem of whether partial integration can hold. Our idea is to  ``cut off'' the control region.  We set the weight in  a small region containing degenerate boundaries to be $0$, so that the integral term on this region is equal to $0$ and hence  we replace the integral over the degenerate region with the integral over only  non-degenerate region  in  the Carleman  estimates. These integral terms include both the internal integral term in $\Omega$ and the boundary integral term on $\partial\Omega$.  We transform these terms into integrals
	over sub-regions $\hat{\Om}$ and $\partial\hat{\Om}$, and then estimate them separately. 
	%	This approach allows us to make the integration by parts meaningful without using the trace theorem. 
	Fortunately,  we can still use our new  Carleman estimates to establish  the null controllability  for  our control system.  Furthermore, we can continue to consider the optimal control problem like \cite{chen2018time, marinoschi2017optimal,zheng2024global}, or combine degradation with stochastic (see \cite{liu2019carleman,wu2020carleman}). 
	The remaining part of this paper are structured as follows. In Section 2, we describe  the main problems and state the key results:  e.g.  Theorem \ref{exist1}, Theorems \ref{TH2}-\ref{TH3},   in particular, in Section 2.2, we provide some specific examples of applications. In Section 3, we provide some auxiliary  results and show  the well-posedness of problem \eqref{1.1} (Theorem \ref{exist1}). In Section 4, we give  Carleman estimates for the degenerate equation \eqref{1.1} and prove Theorems \ref{TH2}-\ref{TH4}.

	\section{main results}
	
	In this section   we first present  our main results for general  degenerate  symmetric 
	parabolic equations.  This result is then applied to  an interesting  specific case which is not covered by earlier results.  
	
	\subsection{General degenerate symmetric form}
	
	\hspace*{\fill}\\
	
	We shall  consider the following parabolic equation  in a bounded domain $\Omega\subset \mathbb{R}^n$ with Lipschitz boundary $\Gamma=\partial \Omega$: 
	\begin{equation}\label{1.1}
		\begin{cases}
			\partial_{t}z - \Div(A\nabla z) =\chi_{\omega_{0}}g, & \mbox{in} \ Q,  \\
			z=0 \ \mbox{or} \ A\nabla z \cdot \nu =0,  & \mbox{on} \ \Sigma,  \\
			z(0)= z_{0},  & \mbox{in} \ \Omega.  
		\end{cases}
	\end{equation}
	Here, $Q:=\Omega\times (0,T)$ and $\Sigma:= \Gamma\times(0,T)$ denote the space-time domain and space time boundary, respectively.  The   initial data $z_0 \in L^{2}(\Omega)$,  $\nu$ is the unit outward normal vector on the boundary.      $\omega_{0}$  is the control domain  and as usual, $\chi_{\omega_{0}}$ denotes  the   indicate  function, and   $g \in L^{2}(Q)$ is the control function.  As we state in introduction  the particularity of our problem is that  part of 
	boundary may not be $C^2$, denote by $\Gamma_1$, and the  coefficient $A$ is allowed to be degenerate  at some sub-region of the boundary, denote by $\Gamma_2$, that is to say, $A(x)=0$ on $\Gamma_2$, $A$ does not satisfy the uniform ellipticity condition on $\Gamma_2$.    To describe the degeneracy that we are going to handle, we use $\Gamma_0\subset \Gamma $ to denote a possibly  degenerate and non-$C^2$ sub-boundary  of $\Gamma$, i.e., $\Gamma_0=\Gamma_1\cup\Gamma_2$. It is known  that the solution $z$ cannot improve regularity on $\Gamma_0$.  The control domain $\omega_{0}$ contains  $ \left\lbrace x\in \Omega \ | \  d(x,\Gamma_0)< \rho\right\rbrace (\subset \omega_{0}$)  for some given constant $\rho >0$.  Let  $\hat{\Omega}$ be a nonempty open subset of $\Omega$ with $C^2$ boundary, satisfying  $\Omega \backslash \bigcup_{x_0\in\Gamma_0} B(x_0, \frac{2\rho}{3}) \subset \hat{\Om} \subset \Omega \backslash \bigcup_{x_0\in\Gamma_0} B(x_0, \frac{\rho}{3})$.  $\hat{\Omega}$ may include partial boundaries but does not contain degenerate boundaries.  Let $\Omega_0$ be an open set with $C^2$ boundary satisfying $\Omega\backslash\omega_{0} \subset \Omega_0 \subset \Omega \backslash \bigcup_{x_0\in\Gamma_0} B(x_0, \frac{2\rho}{3})$.
	%	\footnote{Use one of  $\Omega\backslash\omega_{0}$ or $\Omega- \omega_{0}$ to denote the difference of two sets} 
	Obviously, $\Om_0$ is a proper subset of $\hat{\Om}$. We can refer to  Figure \ref{fig.1}  to visually understand the relationship between $\omega_{0}$, $\hat{\Om}$, and $\Om_0$.
	
	\begin{figure} 
		\begin{center}
			{  \includegraphics[width=0.8\textwidth]{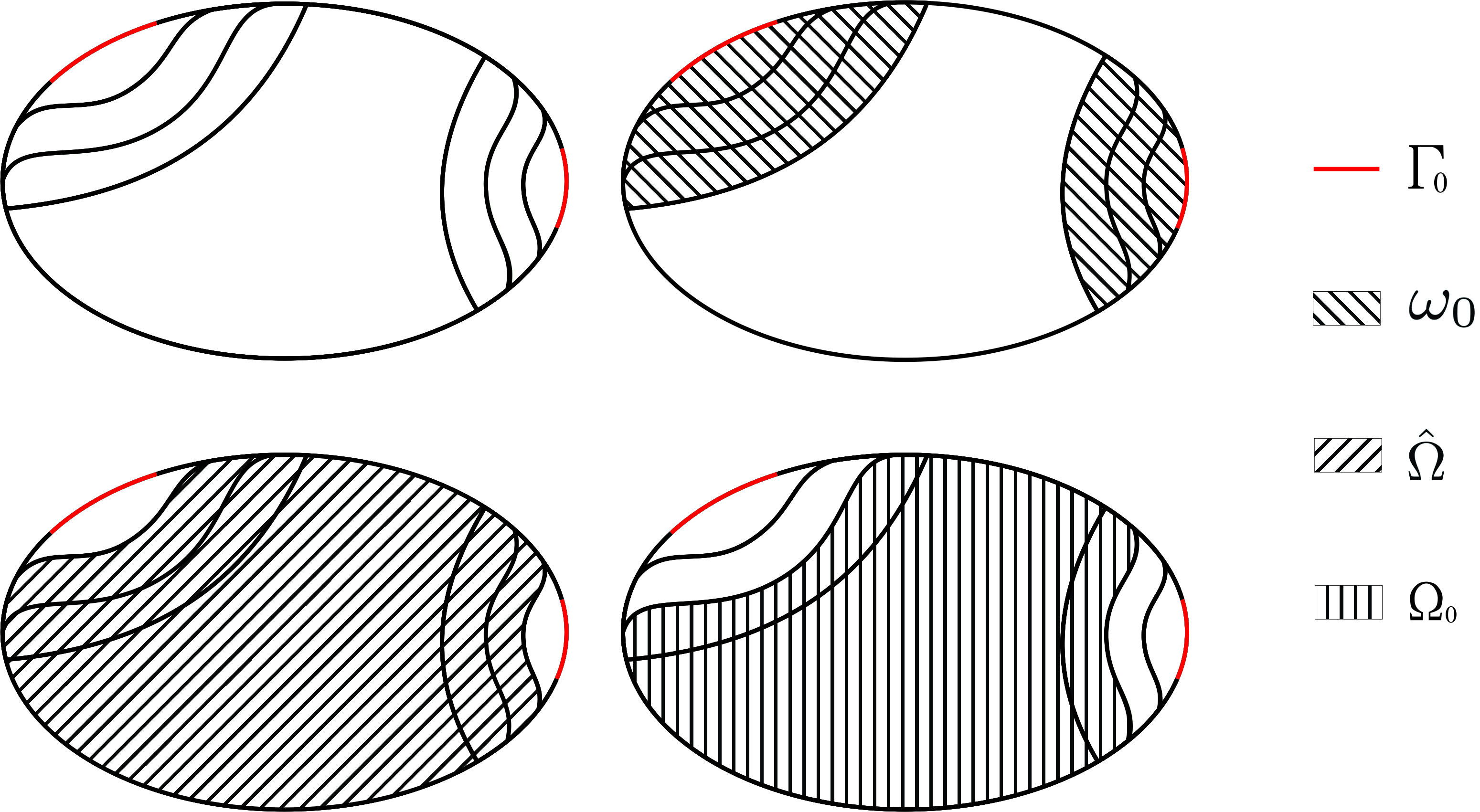}  }
			\renewcommand\figurename{Fig.}
			\caption{Illustration of the sets $\omega_{0}$ and $\hat{\Om}$, and $\Om_0$.
			}\label{fig.1}
		\end{center}
	\end{figure}
	%\begin{minipage}[h]{0.35\linewidth}
	%	\begin{center}
		%	%\includegraphics[height=8\baselineskip]{grafics}
		%	\centerline{\includegraphics[width=2\textwidth]{2.28-1.jpg}}
		%	\end{center}
	%\end{minipage}

	We impose the following assumptions on the matrix-valued function $A=(A_{ij}(x))_{i,j=1}^n$:
	
	\begin{assumption}\label{assume1}
		The matrix-valued function $A=(A_{ij}(x))_{i,j=1}^n : \mathbb{R}^{n} \to \mathbb{R}^{n\times n}$ is measurable, symmetric and positive definite for all $x \in \Om$, and $A\in C^2(\Om)$. It may vanish on  a subset of $\partial\Omega$ ($A=0$ on degenerate boundary),   but 
		on $\hat{\Om}$ it  satisfies the uniform ellipticity condition: 
		$$
		\ \beta|\xi|^2 \le \sum_{i,j=1}^n A_{ij}(x)\xi_i\xi_j \le \Lambda |\xi|^2, \quad  \forall \xi\in\mathbb{R}^n, \  \forall x \in \hat{\Om}, 
		$$
		for certain   two fixed constants $0<\beta \le \Lambda<\infty$.
	\end{assumption}
	From Assumption \ref{assume1}, it is known that the equation \eqref{1.1} is degenerate on a part of the boundary, but non-degenerate in the interior of $\Omega$.
	
	Below, we introduce two solution spaces, denoted by $\mathcal{H}^1(\Omega)$ and $\mathcal{H}^2(\Omega)$, corresponding respectively to the Dirichlet boundary condition   and Neumann boundary condition:
	\begin{equation}\label{space}
		\begin{split}
			\mathcal{H}^1(\Om)
			&=\left\lbrace z\in L^2(\Om) \mid \nabla z \cdot A \nabla z \in L^1(\Om) \right\rbrace,\\
			\mathcal{H}^2(\Om)
			&=\left\lbrace z\in \mathcal{H}^1(\Om) \mid \Div(A\nabla z) \in L^2(\Om) \right\rbrace.
		\end{split}
	\end{equation}
	These spaces are equipped with the following scalar products:
	\begin{equation}\label{inner}
		\begin{split}
			(z,v)_{\mathcal{H}^1(\Om)}
			&:=\int_{\Omega} zv  dx  + \int_{\Omega} \nabla z \cdot A \nabla v dx,\\
			(z,v)_{\mathcal{H}^2(\Om)}
			&:=\int_{\Omega} zv  dx  + \int_{\Omega} \nabla z \cdot A \nabla v dx + \int_{\Omega} \Div(A\nabla z)\Div(A\nabla v) dx 
		\end{split}
	\end{equation}
	or in terms of the corresponding Hilbert  norms:
	\begin{equation}\label{norm}
		\begin{split}
			\left\| z \right\|_{\mathcal{H}^1(\Om)} 
			&= \left\| z \right\|_{L^2(\Om)} + \left\| \nabla z \cdot A \nabla z\right\|_{L^1(\Om)},\\
			\left\| z \right\|_{\mathcal{H}^2(\Om)} 
			&= \left\| z \right\|_{\mathcal{H}^1(\Om)} + \left\| \Div(A\nabla z)\right\|_{L^2(\Om)}.
		\end{split}
	\end{equation}
	It can be easily verified that $(\mathcal{H}^1(\Omega), (\cdot,\cdot)_{\mathcal{H}^1(\Omega)})$ and $(\mathcal{H}^2(\Omega), (\cdot,\cdot)_{\mathcal{H}^2(\Omega)})$ are Hilbert  spaces.
	%	, and $(\mathcal{H}^1(\Omega), |\cdot|_{\mathcal{H}^1(\Omega)})$ and $(\mathcal{H}^2(\Omega), |\cdot|_{\mathcal{H}^2(\Omega)})$ are Banach spaces.  
	
	Let $\mathcal{H}_0^1(\Omega)$ denote the closure of $C_0^\infty(\Omega)$ in the space $\mathcal{H}^1(\Omega)$. In other words,
	\begin{equation*}
		\mcH_0^1(\Omega)=\overline{C_0^\infty(\Omega)}^{\mcH^1(\Omega)}.
	\end{equation*} 
	We also introduce the operators $(\mathcal{A}_1,D(\mathcal{A}_1))$ and $(\mathcal{A}_2,D(\mathcal{A}_2))$ corresponding respectively to the Dirichlet  and Neumann boundary  problems   as follows:
	\begin{equation}
		\left\{	\begin{split}
			&\mathcal{A}_1 z= \sum_{i,j=1}^n\partial_{x_i}(A_{ij}\partial_{x_j}z),  \quad z\in D(\mathcal{A}_1) = \mathcal{H}^2 (\Om)\cap \mcH_0^1(\Omega) ,\\
			&\mathcal{A}_2 z= \sum_{i,j=1}^n\partial_{x_i}(A_{ij}\partial_{x_j}z),  \quad z\in  D(\mathcal{A}_2) = \left\lbrace z\in \mathcal{H}^2 (\Om) \mid A\nabla z \cdot \nu |_{\Sigma} = 0 \right\rbrace .
		\end{split}\right. \label{e.2.5} 
	\end{equation}
	
	The first result of this paper is about the well-posedness of	equation \eqref{1.1} and an estimate of its solution, whose proof is given in next section. 
	\begin{theorem}\label{exist1}
		For any $g \in L^2(Q)$ and any $z_0 \in L^2(\Omega)$, there exists a unique solution $z \in C^0([0,T];L^2(\Omega)) \cap L^2(0,T;\mathcal{H}_0^1(\Omega))$ to equation \eqref{1.1} with homogeneous Dirichlet boundary condition. Moreover, there exists a positive constant $C$ such that
		$$
		\sup_{t\in\left[0,T\right] } \|z(t)\|_{L^2(Q)}^2 + \int_{0}^{T} \left\| z(t)\right\|^2_{\mcH_0^1(\Omega)} d t \le C(\left\| z_0\right\|_{L^2(\Om)}^2 + \|\chi_{\omega_{0}} g\|_{L^2(Q)}^2).
		$$
		Similarly, for any $g \in L^2(Q)$ and any $z_0 \in L^2(\Omega)$, there exists a unique solution $z \in C^0([0,T];L^2(\Omega)) \cap L^2(0,T;\mathcal{H}^1(\Omega))$ to equation \eqref{1.1} with homogeneous Neumann boundary condition. Furthermore, there exists a positive constant $C$ such that
		$$
		\sup_{t\in\left[0,T\right] } \|z(t)\|_{L^2(Q)}^2 + \int_{0}^{T} \left\| z(t)\right\|^2_{\mcH^1(\Omega)} d t \le C(\left\| z_0\right\|_{L^2(\Om)}^2 + \|\chi_{\omega_{0}} g\|_{L^2(Q)}^2).
		$$
	\end{theorem} 
	%\begin{theorem}\label{exist2}
	%	For any $g \in L^2(Q)$ and any $z_0 \in L^2(\Omega)$, there exists a unique solution $z \in C^0([0,T];L^2(\Omega)) \cap L^2(0,T;\mathcal{H}^1(\Omega))$ to equation \eqref{1.1} with homogeneous Neumann boundary conditions. Furthermore, there exists a positive constant $C$ such that
	%	$$
	%	\sup_{t\in\left[0,T\right] } \|z(t)\|_{L^2(Q)}^2 + \int_{0}^{T} \left\| z(t)\right\|^2_{\mcH^1(\Omega)} d t \le C(\left\| z_0\right\|_{L^2(\Om)}^2 + \|\chi_{\omega_{0}} g\|_{L^2(Q)}^2).
	%	$$
	%\end{theorem}
	\begin{remark}
		It is worth noting that by adding a gradient term $c\nabla z$ to the left-hand side of equation \eqref{1.1} and imposing stronger conditions on the degenerate coefficient $A_{ij}(x)$, similar well-posedness results can also be obtained.
	\end{remark}
	
	%\subsection{Internal regularity improvement}
	%Now, we have proved the existence of solutions to the equation \eqref{1.1}, and the solution $z \in C([0,T];L^{2}(\Omega))\cap L^{2}(0,T;\mathcal{H}_0^{1}(\Omega))$.
	%We are going to show that the regularity of the solutions can be improved in the interior of $\Omega$. By standard argument in \cite{EVANS}, one can prove the following result.
	%\begin{theorem}\label{Regularity}
	%For all the solution of equation \eqref{1.1}, one has $A\nabla u \in W^{1,1}(\Om_0)$ and $\nabla u A\nabla u \in W^{1,1}(\Om_0)$.
	%\end{theorem}
	%\begin{proof}
	%By standard argument in \cite{EVANS}, one can prove this result.
	%\end{proof}
	
	%\begin{remark}
	%It is worth noting that, on the boundary, we cannot improve the regularity due to the specificity of the equation we are considering. Based on this fact, on the boundary, we do not have the fact $(A \nabla u \cdot \nabla \sigma) A \nabla u, (A\nabla u\cdot\nabla u)A\nabla\sigma\in (W^{1,1}(\Omega))^2$ (see the estimate of $I_4$ about the boundary term in the following section), in other words, the functions $(A \nabla u \cdot \nabla \sigma) A \nabla u, (A\nabla u\cdot\nabla u)A\nabla\sigma$ do not have trace on $\partial\Omega$, which makes our research approach very different from  \cite{CA5,cannarsa2016global}, and this is also the reason that we choose the control domain is $\omega_{0}:= \Om\backslash \overline{\Om_0}$ in the following section.
	%\end{remark}
	
	%\section{Carleman esitimate and Null Controllability results}
	
	%\subsection{main results}
	Let us consider the adjoint equation corresponding to \eqref{1.1}:
	%Obviously, the adjoint equation of \eqref{1.1} is
	\begin{equation}\label{3.1}
		\begin{cases}
			\partial_{t}w + \Div(A\nabla w)=f, & \mbox{in} \ Q,  \\
			w=0 \ \mbox{or} \ A\nabla w \cdot \nu =0,  & \mbox{on} \ \Sigma,  \\
			w(  T)= w_T,  & \mbox{in} \ \Omega,  
		\end{cases}
	\end{equation}
	where $f\in L^2(Q)$ and $w_T \in L^2(\Omega)$. The main results of this paper are the Carleman inequality and observability inequality stated below.  To state the Carleman inequality, let us 
	consider a function $\eta$ satisfying
	\begin{equation}\label{ETA1}
		\eta \in C_0^3(\Omega), \quad
		\eta(x)  
		\begin{cases}
			=0, & x \in \Omega\backslash\hat{\Om},\\
			>0, & x \in \hat{\Om},
		\end{cases}
	\end{equation}
	and
	\begin{equation}\label{ETA2}
		\quad \left\| \nabla \eta \right\|_{L^2(\Om)} \ge C >0 \ \mbox{in} \ \Om_0, \quad  \nabla \eta=0 \ \mbox{on} \ \partial \hat{\Om}- \partial \hat{\Om}\cap \partial \Om.
	\end{equation}
	Define
	\begin{equation}
		\begin{split}
			& \theta(t):=[t(T-t)]^{-4}, \quad \xi(x, t):=\theta(t) e^{ \lambda(8|\eta|_\infty+\eta (x))}, \quad \sigma(x, t):=\theta(t) e^{10 \lambda|\eta|_\infty}-\xi(x, t).
		\end{split}\label{e.2.7} 
	\end{equation}
	In what follows, $C>0$ represents a generic constant, and $w$ denotes a solution of equation \eqref{3.1}. We can assume, using standard arguments, that $w$ possesses sufficient regularity. Specifically, we consider $w \in H^1(0,T; \mathcal{H}_0^1(\Omega))$ with homogeneous Dirichlet boundary conditions and $w \in H^1(0,T; \mathcal{H}^1(\Omega))$ with homogeneous Neumann boundary conditions.
	In the case of Dirichlet boundary conditions, i.e., for the equation
	\begin{equation}\label{Dirichlet}
		\begin{cases}
			\partial_{t}w + \Div(A\nabla w)=f, & \mbox{in} \ Q,  \\
			w=0,  & \mbox{on} \ \Sigma,  \\
			w(  T)= w_T,  & \mbox{in} \ \Omega,  
		\end{cases}
	\end{equation}
	we have the following Carleman inequality.
	\begin{theorem}\label{TH2}
		There exist positive constants $C, s_0, \lambda_0$ such that for any $\lambda \ge \lambda_0$, $s \ge s_0$, and any solution $w$ to \eqref{Dirichlet}, the following inequality holds:
		\begin{equation}\label{3.2}
			\begin{split}
				&s^{-1}\iint_{Q} \xi^{-1} (|\partial _t w  |^2 + \left|\Div(A\nabla w)\right|^2) dx dt \\
				&+C\iint_Q s^3 \lambda^4 \xi^3\left|A \nabla \eta \cdot \nabla \eta \right|^2|w|^2 dx dt  + C s  \lambda^2 \iint_Q\xi|\nabla w \cdot A \nabla \eta|^2 dx  d t\\
				&\qquad\quad \leq  C\left\|e^{-s \sigma} f\right\|^2
				+Cs^3 \lambda^3 \int_0^T \int_{\omega_{0}} \xi^3|w|^2dx  d t.
			\end{split}
		\end{equation}
	\end{theorem}
	However, in the case of Neumann boundary conditions, i.e., for the equation
	\begin{equation}\label{Neumann}
		\begin{cases}
			\partial_{t}w + \Div(A\nabla w) =f, & \mbox{in} \ Q,  \\
			A\nabla w\cdot \nu=0,  & \mbox{on} \ \Sigma,  \\
			w(  T)= w_T,  & \mbox{in} \ \Omega,  
		\end{cases}
	\end{equation}
	the Carleman inequality takes a slightly different form.  
	\begin{theorem}\label{TH22}
		There exist positive constants $C, s_0, \lambda_0$ such that for any $\lambda \ge \lambda_0$, $s \ge s_0$, and any solution $w$ to \eqref{Neumann}, the following inequality holds:
		\begin{equation}\label{CarlemanN}
			\begin{split}
				C\iint_Q&  \left(  s^{-1} \xi^{-1} \left( |\partial _t w|^2 + |\Div(A\nabla w)|^2\right)  + s\lambda^2 \xi|A\nabla w \cdot \nabla \eta|^2 + s\lambda^2 \xi A\nabla w \cdot \nabla w\right. \\
				&\left. \qquad+ s^3 \lambda^4 \xi^3 w^2 dx \right) \left(e^{2s\sigma} + e^{2s\wt{\sigma}} \right)    d t\\%3
				\le& C \iint_Q |f|^2\left(e^{2s\sigma} + e^{2s\wt{\sigma}} \right) dx dt 
				+ C\int_{0}^{T}\int_{\om_0} s^3\lambda^4\xi^3 w^2 \left(e^{2s\sigma} + e^{2s\wt{\sigma}} \right) dxdt.
			\end{split}
		\end{equation}
	\end{theorem}
	By a standard argument, we can conclude the following theorem. 
	\begin{theorem}\label{TH4}
		For a fixed $T>0$ and an open set $\omega_{0} \subset \Omega$ as defined previously, assuming that \eqref{3.2} and \eqref{CarlemanN} holds, there exists a positive constant $C>0$ such that for any $w_T \in L^2(\Omega)$, the solution to \eqref{3.1} satisfies
		\begin{equation}\label{3.4}
			\int_{\Omega}|w(x, 0)|^2 d x  \leq C \iint_{\omega_{0} \times(0, T)}|w|^2 d x d t.
		\end{equation}
	\end{theorem}
	The proofs of Theorems \ref{TH2}-\ref{TH4} will be provided in Chapter 4.  Since  proving controllability is equivalent to establishing an observability property for the adjoint system \eqref{3.1} (see \cite{lions1992remarks,rockafellar1967duality})   we can now  use  the duality between controllability and   observability to obtain  easily the null controllability of \eqref{1.1}.
	\begin{theorem}\label{TH3}
		For a fixed $T>0$ and an open set $\omega_{0} \subset \Omega$ as defined previously, assuming that \eqref{3.2} holds, there exists a control $g \in L^2(Q)$ such that the solution $z$ of \eqref{1.1} satisfies
		$$
		z(\cdot, T)=0 \mbox{ in } \Omega.
		$$
		%{\red 		Moreover, there exists a constant $C>0$ such that
			%		$$
			%		\|g\|_{L^2(\Om)} \leq C\left|z_0\right| .
			%		$$
			%	}
	\end{theorem}
	
	\subsection{Specific examples} 
	\hspace*{\fill}\\
	
	In order to better illustrate the results of this paper, we provide several specific examples of applications. 
	
	\noindent {\bf Example 1}.   First, we consider a specific example of a two-dimensional degenerate parabolic equation:
	\begin{equation}\label{ex1}
		\begin{cases}
			\partial_{t}z - (y^{\alpha_2}\partial_{xx}z +x^{\alpha_1}\partial_{yy}z )=\chi_{\omega_{0}}g, & \mbox{in} \ Q,  \\
			z(x,y,t)=0 \ \mbox{or} \ A\nabla z \cdot \nu =0, & \mbox{on} \ \Sigma,  \\
			z(x,  y,  0)= z_{0}(x,  y),  & \mbox{in} \ \Omega,  
		\end{cases}
	\end{equation}
	where $\alpha_1$, $\alpha_2 \in (0,2)$,   $\Omega=(0,1)\times(0,1)$, $\Gamma:= \partial\Omega$, $T>0$, $Q:=\Omega\times (0,T)$, $\Sigma:= \Gamma\times(0,T)$. It is easy to see that the degenerate boundary and non-$C^2$ boundary consist  of  $\Gamma_0=\{0\}\times(0,1)\cup (0,1)\times\{0\}\cup (1,1)$,  the vertex $(1,1)$ at the top right corner of $\Omega$.  Let $\omega_{0} = \left\lbrace x\in \Omega \ | \  d(x,\Gamma_0)< \rho\right\rbrace$.
	%,  and $\chi_{\omega_{0}}$ is the corresponding characteristic function. 
	The control $g \in L^{2}(Q)$ and $z_0 \in L^{2}(\Omega)$ are given. The selections of $\hat{\Om}$ and $\Om_0$ are    as before.
	
	\begin{figure} 
		\begin{center}
			{  \includegraphics[width=0.8\textwidth]{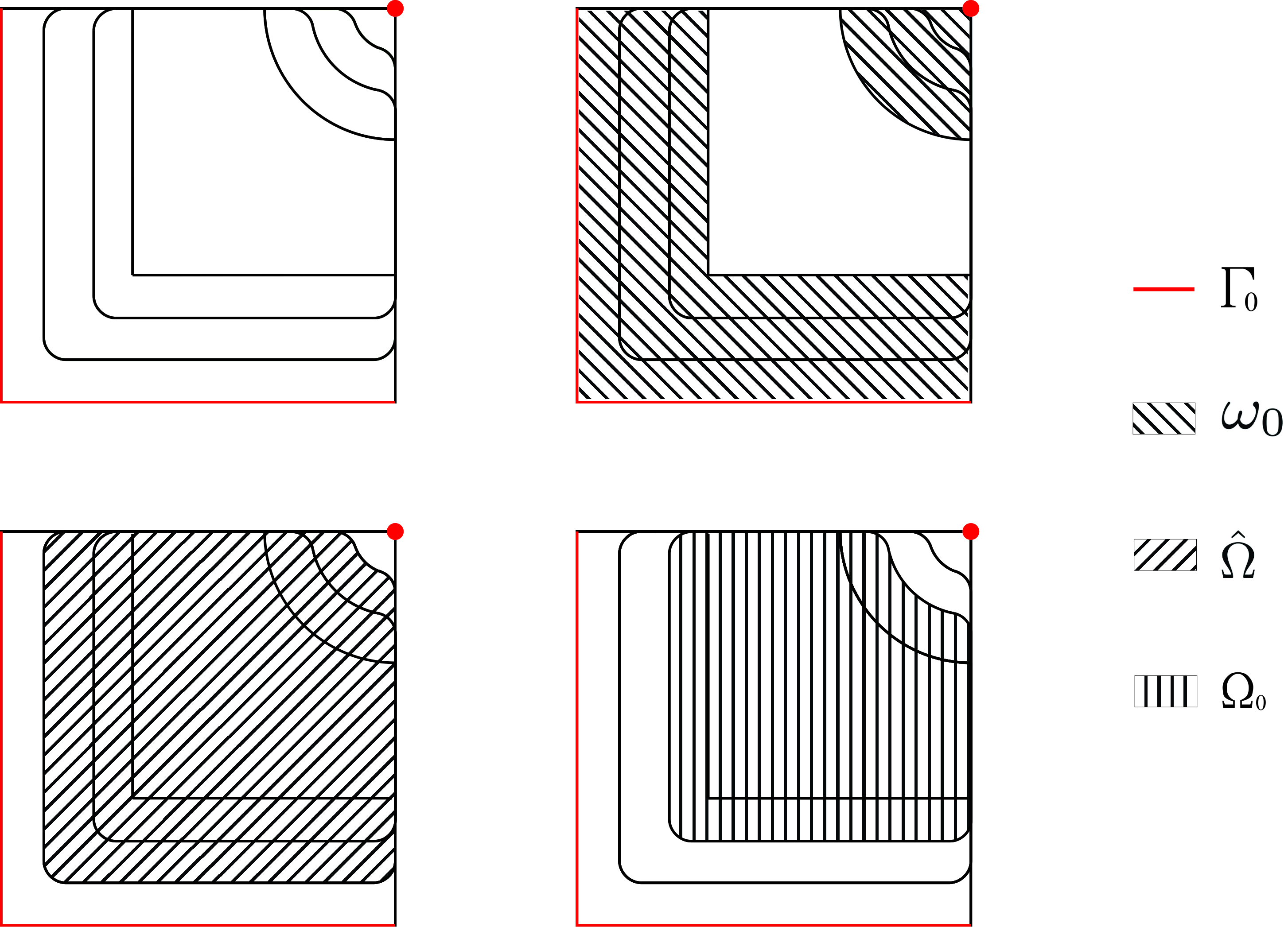}  }
			\renewcommand\figurename{Fig.}
			\caption{Illustration of the sets $\omega_{0}$ and $\hat{\Om}$, and $\Om_0 $ in specific case.
			}\label{fig.2}
		\end{center}
	\end{figure} 
	%	The function space, inner product, and norm are defined as in previous subsection, 
	%	%equations \eqref{space}-\eqref{norm}, 
	%	and t
	The matrix-valued function $A:\overline{\Omega} \to M_{2\times 2}(\mathbb{R})$ is given by
	\begin{equation*}
		\begin{pmatrix} y^{\alpha_2} & 0 \\ 0 &x^{\alpha_1} \end{pmatrix}.
	\end{equation*} 
	%	By substituting the above example into the general form introduced earlier, 
	It is easy to verify Assumption \ref{assume1} and we can then  obtain the controllability of the equation \eqref{ex1}.
	\begin{proposition}
		For a fixed $T>0$ and an open set $\omega_{0} \subset \Omega$ as defined, there exists a control $g \in L^2(Q)$ such that the solution $z$ of \eqref{ex1} satisfies
		$$
		z(\cdot, T)=0 \mbox{ in } \Omega.
		$$
		%		Moreover, there exists a constant $C=C(T, \omega_{0})>0$ such that
		%		$$
		%		\|g\|_{L^2(\Om)} \leq C\left|z_0\right| .
		%		$$
	\end{proposition}
	%	That is to say, for a fixed $T>0$ and an open set $\omega_{0} \subset \Omega$ as defined above, there exists a control $g \in L^2(Q)$ such that the solution $z$ of \eqref{4.1} satisfies
	%		$$
	%		z(\cdot, T)=0 \mbox{ in } \Omega.
	%		$$ 

	\noindent{\bf Example 2}. 	Similarly,we can also consider 
	\begin{equation}\label{ex2}
		\begin{cases}
			\partial_{t}z - \Div(A\nabla z)=\chi_{\omega_{0}}g, & \mbox{in} \ Q,  \\
			z(x,y,t)=0 \ \mbox{or} \ A\nabla z \cdot \nu =0, & \mbox{on} \ \Sigma,  \\
			z(x,  y,  0)= z_{0}(x,  y),  & \mbox{in} \ \Omega,  
		\end{cases}
	\end{equation}
	All other conditions remain unchanged as for equation \eqref{ex1}, except now 
	%	that $A:\overline{\Omega} \to M_{2\times 2}(\mathbb{R})$ has been replaced by
	\begin{equation*}
		A=	\begin{pmatrix} x^{\alpha_1} & 0 \\ 0 & y^{\alpha_2} \end{pmatrix}:  \overline{\Omega} \to M_{2\times 2}(\mathbb{R})\, .
	\end{equation*}
	This equation  is  studied in \cite{BSV}, but our method is different  and applicable to a wider range.  In this case, $\Gamma_0$ remains as $\{0\}\times(0,1)\cup (0,1)\times\{0\}\cup (1,1)$, and we  can also easily establish Carleman estimates to get null controllability.
	\begin{proposition} 
		For a fixed $T>0$ and an open set $\omega_{0} \subset \Omega$ as defined previously, there exists a control $g \in L^2(Q)$ such that the solution $z$ of \eqref{ex2} satisfies
		$$
		z(\cdot, T)=0 \mbox{ in } \Omega.
		$$
		%		Moreover, there exists a constant $C=C(T, \omega_{0})>0$ such that
		%		$$
		%		\|g\|_{L^2(\Om)} \leq C\left|z_0\right| .
		%		$$
	\end{proposition}
	It is worth noting that when we replace the above domain $\Omega$ with $\{(x,y)|x^2+y^2\le 1,x>0,y>0\}$ in the above two cases  \eqref{ex1} and \eqref{ex2}, at this time the degenerate boundary $\Gamma_0$ becomes  only  $\{0\}\times(0,1)\cup (0,1)\times\{0\}$,  making thing even  simpler. 
	
	\noindent{\bf Example 3}.  The above two examples are two dimensional equations. Now we give a four dimensional example as follows: 
	\begin{equation}\label{ex3}
		\begin{cases}
			\partial_{t}z - \Div(A\nabla z)=\chi_{\omega_{0}}g, & \mbox{in} \ Q,  \\
			z(x,y,z,w,t)=0 \ \mbox{or} \ A\nabla z \cdot \nu =0, & \mbox{on} \ \Sigma,  \\
			z(x,  y,z,w,  0)= z_{0}(x,  y,z,w,),  & \mbox{in} \ \Omega,  
		\end{cases}
	\end{equation}
	where   $\Omega=(0,1)\times(0,1)\times(0,1)\times(0,1)$, $\Gamma:= \partial\Omega$, $T>0$, $Q:=\Omega\times (0,T)$, $\Sigma:= \Gamma\times(0,T)$. It is easy to see that the degenerate boundary and non-$C^2$ boundary are  $\Gamma_0=\{0\}\times(0,1)\times(0,1)\times(0,1)\cup (0,1)\times\{0\}\times(0,1)\times(0,1)\cup(0,1)\times(0,1)\times\{0\}\times(0,1)\cup (0,1)\times(0,1)\times(0,1)\times\{0\}\cup (1,1,1,1)$. Let $\omega_{0} = \left\lbrace x\in \Omega \ | \  d(x,\Gamma_0)< \rho\right\rbrace$,  and $\chi_{\omega}$ be the corresponding characteristic function. The control $g \in L^{2}(Q)$ and $z_0 \in L^{2}(\Omega)$ are given. The selection of $\hat{\Om}$ and $\Om_0$ are as before 
	and the matrix-valued function $A:\overline{\Omega} \to M_{4\times 4}(\mathbb{R})$ is given by
	\begin{equation*}
		\begin{pmatrix} y^{\alpha_2} & 0 &0 &0 \\ 0 & x^{\alpha_1} &0 &0 \\ 0 & 0 & z^{\alpha_3} &0 \\ 0 & 0 & 0 & w^{\alpha_4}  \end{pmatrix}\,,
	\end{equation*}
	where $\alpha_1$, $\alpha_2$, $\alpha_3$ and $\alpha_4$  $\in (0,2)$.  	It is easy to verify the Assumption 2.1 in this case and  Theorem 
	\ref{TH3}
	can  also be applied. 
	\begin{proposition}
		For a fixed $T>0$ and an open set $\omega_{0} \subset \Omega$ as defined previously, there exists a control $g \in L^2(Q)$ such that the solution $z$ of \eqref{ex3} satisfies
		$$
		z(\cdot, T)=0 \mbox{ in } \Omega.
		$$
		%		Moreover, there exists a constant $C=C(T, \omega_{0})>0$ such that
		%		$$
		%		\|g\|_{L^2(\Om)} \leq C\left|z_0\right| .
		%		$$
	\end{proposition}
	
	\section{Well-posed results}
	In this section, we investigate the well-posedness of \eqref{1.1}. 
	Let   $\mathcal{A}_1 $, $ \mathcal{A}_2$, $D(\mathcal{A}_1)$,  and $D(\mathcal{A}_2)$  be defined as in \eqref{e.2.5} and let $i_1f =i_2f=f$ be the identity operators 
	(defined on different domains $D(\mathcal{A}_1)$,$D(\mathcal{A}_2)$).  
	\begin{lemma}\label{2.15}
		The following results hold:
		\begin{itemize}
			\item [(1)] The injection $i_1: D(\mathcal{A}_1)\to L^2(\Om)$ and $i_2: D(\mathcal{A}_2)\to L^2(\Om)$ are continuous with dense range.
			\item [(2)] The bilinear form $q_1(u,v):= \int_\Om \nabla u \cdot A \nabla v dx, \ (u,v)\in D(\mathcal{A}_1)\times \mcH^1_0(\Omega)$ and $q_2(u,v):= \int_\Om \nabla u \cdot A \nabla v dx, \ (u,v)\in D(\mathcal{A}_2)\times \mcH^1(\Omega),$ are continuous, symmetric.
			\item [(3)]  The operators $(\mathcal{A}_1,D(\mathcal{A}_1))$ and $(\mathcal{A}_2,D(\mathcal{A}_2))$ are self-adjoint and dissipative.
		\end{itemize}
		%are the infinitesimal generators of the strongly continuous semigroups denoted by $e^{t\mathcal{A}_1}, e^{t\mathcal{A}_2}$ respectively.
	\end{lemma}
	\begin{proof}
		Clearly, the continuity of $i_1$ and $i_2$ follows directly from the definition of $D(\mathcal{A}_1)$ and $D(\mathcal{A}_2)$. Consequently, since $C_0^\infty (\Om)$ is dense in $D(\mathcal{A}_1)$, it is also dense in $L^2(\Om)$. Moreover, since $C_0^\infty (\Om)\subset \mathcal{H}^2(\Om)\subset L^2(\Om)$, we have $\mathcal{H}^2(\Om)$ is dense in $L^2(\Om)$. This establishes (1).
		
		Next, we prove (2). It is evident that $q_1$ and $q_2$ are symmetric, and $q_1(u,u) \ge 0$ and $q_2(u,u) \ge 0$ for all $u$. Furthermore,
		$$
		|q_1(u,v)|\le\left\| u\right\| _{\mcH_0^1(\Omega)} \left\| v\right\| _{\mcH_0^1(\Omega)},\ |q_2(u,v)|\le\left\| u\right\| _{\mcH^1(\Omega)} \left\| v\right\| _{\mcH^1(\Omega)}.
		$$
		Hence, $q_1$ and $q_2$ are continuous    and monotone. 
		
		Now, let $(u_1,v_1) \in D(\mathcal{A}_1)\times \mcH^1_0(\Omega)$ and $(u_2,v_2)\in D(\mathcal{A}_2)\times \mcH^1(\Omega)$.   Applying Green's  formula, we have
		$$
		q_1(u_1,v_1)=-\int_{\Om}\mathcal{A}_1 u_1 \cdot v_1 d x, \ q_2(u_2,v_2)=-\int_{\Om}\mathcal{A}_2 u_2 \cdot v_2 d x.
		$$
		From Lax-Milgram theorem we can deduce that $\mathcal{A}_1$ and $\mathcal{A}_2$  are  maximal monotone. Hence,   $\mathcal{A}_1$ and $\mathcal{A}_2$ are self-adjoint.
		Furthermore, since
		$$
		\begin{aligned}
			\left\| (\gamma I - \mathcal{A}_1)u_1\right\| _{L^2(\Om)}
			&=\sup_{\|v_1\|_{L^2(\Om)} \le 1, v_1 \in \mcH_0^1(\Omega)} \int_{\Omega} \gamma v_1 u_1 - v_1 \mathcal{A}_1 u_1 dx
			%=&\sup_{|v|_2 \le 1, v \in \mcH_0^1(\Omega)} \int_{\Omega} \gamma v u dx + q(u,v)\\
			\ge  \gamma \left\| u_1\right\| _{L^2(\Om)},\\
			\left\| (\gamma I - \mathcal{A}_2)u_2\right\| _{L^2(\Om)}
			&=\sup_{\|v_2\|_{L^2(\Om)} \le 1, v_2 \in \mcH^1(\Omega)} \int_{\Omega} \gamma v_2 u_2 - v_2 \mathcal{A}_2 u_2 dx
			%=&\sup_{|v|_2 \le 1, v \in \mcH^1(\Omega)} \int_{\Omega} \gamma v u dx + q(u,v)\\
			\ge  \gamma \left\| u_2\right\| _{L^2(\Om)},
		\end{aligned}
		$$
		we conclude that $\mathcal{A}_1$ and $\mathcal{A}_2$ are dissipative. 
	\end{proof}
	
	\begin{proof}[{\bf Proof of Theorem \ref{exist1}}].  By above Lemma \ref{2.15}   both $\mathcal{A}_1$ and $\mathcal{A}_2$ can serve as   infinitesimal generators of   strongly continuous semigroups denoted by $e^{t\mathcal{A}_1}$ and $e^{t\mathcal{A}_2}$, respectively.
		%		 Additionally, the family of operators in $\mathcal{L}(L^2(\Om))$ given by
		%		$$
		%		B(t)u := b(t,\cdot)u, \quad t\in (0,T), \quad u\in L^2(\Om)
		%		$$
		%		can be regarded as a family of bounded perturbations of $\mathcal{A}_1$ (resp. $\mathcal{A}_2$).
		Consequently, utilizing the standard techniques (e.g.  \cite{bensoussan2007representation,cazenave1998introduction,showalter1995hilbert}), one can establish the validity of Theorem \ref{exist1}.
		%	We have now demonstrated the existence of solutions to equation \eqref{1.1}, with the solution $z \in C([0,T];L^{2}(\Omega))\cap L^{2}(0,T;\mathcal{H}_0^{1}(\Omega))$.
	\end{proof} 
	%	 Our next goal  is to establish that the regularity of the solutions can be enhanced within the interior of $\Omega$. By employing standard arguments found in \cite{EVANS}, we can easily to deduce that $A\nabla u \in W^{1,1}(\hat{\Om})$ and $\nabla u \cdot A\nabla u \in W^{1,1}(\hat{\Om})$.
	
	\section{Carleman estimates}
	Let us now derive the  Carleman estimate. We begin by considering the  Dirichlet boundary condition.
	
	\begin{proof}[{\bf Proof of Theorem \ref{TH2}}]	Let $\omega$ be a nonempty open set satisfying $\omega\subset\subset \om_0$. Recall $\xi$ is defined in \eqref{e.2.7}. Consider $s>s_0>0$ and introduce
		\begin{equation}\label{Defineu}
			u(x,t)=e^{-s\sigma(x,t)} w(x,t), 
		\end{equation}
		where $w$ is the solution of equation \eqref{Dirichlet}.
		It is easy to demonstrate    the following properties   for $u$:
		\begin{itemize}
			\item [($i$)] $u=\frac{\partial u}{\partial x_i}=0$ at  $t=0$ and $t=T$; 
			\item [($ii$)] $u=0$  on $\Sigma$;
			\item [($iii$)] Let 
			\begin{equation*}
				\begin{split}
					P_1 u
					&:=-2s\lambda^2\xi u  A \nabla \eta\cdot \nabla \eta- 2s\lambda\xi   A \nabla u\cdot \nabla \eta +u_t=:(P_1u)_1+(P_1u)_2+(P_1u)_3,\\
					P_2 u
					&:=s^2\lambda^2\xi^2 u  A \nabla \eta\cdot \nabla \eta+\Div( A \nabla u)+s \sigma_t u=:(P_2u)_1+(P_2u)_2+(P_2u)_3,\\
					f_{s,\lambda}
					&:=e^{-s \sigma} f - s\lambda^2\xi u  A \nabla \eta\cdot \nabla \eta + s\lambda \xi u\Div( A \nabla \eta) .
				\end{split}
			\end{equation*}
			Then 
			\begin{equation*}
				P_1 u+P_2 u=f_{s,\lambda}.
			\end{equation*}
		\end{itemize}
		Taking the $L^2(Q)$ norm of the above identity and using  ($iii$),  we see    
		\begin{equation}\label{PPP}
			\left\|P_1 u\right\|_{L^2(Q)}^2+\left\|P_2 u\right\|_{L^2(Q)}^2+2\sum_{i,j=1}^{3}\left((P_1 u)_i, (P_2 u)_j\right)_{L^2(Q)}=\left\|f_{s,\lambda}\right\|_{L^2(Q)}^2.
		\end{equation} 
		First, we have
		\begin{equation}\label{T1}
			\left((P_1 u)_1, (P_2 u)_1\right)_{L^2(Q)}=-2s^3 \lambda^4 \iint_Q \left|  A \nabla  \eta \cdot \nabla \eta\right| ^2 \xi^3 u^2 dx  dt=:T_1.
		\end{equation}
		By \eqref{ETA1}, $\eta=0$ outside of $\hat{\Om}$, we can deduce
		\begin{equation*}
			\begin{split}
				\left((P_1 u)_2, (P_2 u)_1\right)_{L^2(Q)}=
				&-2s^3 \lambda^3 \iint_Q \xi^3  (A \nabla  \eta \cdot \nabla \eta)   (A \nabla \eta  \cdot \nabla u )u dx  dt\\%1
				=& -   s^3 \lambda^3 \iint_Q \xi^3  (A \nabla  \eta \cdot \nabla \eta)   A  \nabla \eta \cdot \nabla (u^2) dx  dt\\%21
				=& - s^3 \lambda^3 \int_{0}^{T}\int_{\hat{\Om}} \xi^3  (A \nabla  \eta \cdot \nabla \eta)   A  \nabla \eta \cdot \nabla (u^2) dx  dt.
			\end{split}
		\end{equation*}
		Thus
		\begin{equation*}
			\begin{split}
				\left((P_1 u)_2, (P_2 u)_1\right)_{L^2(Q)}
				%				=& -2s^3 \lambda^3 \int_{0}^{T}\int_{\hat{\Om}} \xi^3  (A \nabla  \eta \cdot \nabla \eta)   A  \nabla \eta \cdot\nabla (u^2) dx  dt\\%1
				=&-s^3 \lambda^3 \int_{0}^{T}\int_{\partial\hat{\Om}} \xi^3 u^2 (A \nabla  \eta \cdot \nabla \eta)   (A  \nabla \eta \cdot \nu)  ds  dt\\%2
				&\qquad +s^3 \lambda^3 \int_{0}^{T}\int_{\hat{\Om}} \Div(\xi^3  (A \nabla  \eta \cdot \nabla \eta)   A  \nabla \eta) u^2 dx  dt\\ 
				=& s^3 \lambda^3 \int_{0}^{T}\int_{\hat{\Om}} \Div(\xi^3  (A \nabla  \eta \cdot \nabla \eta)   A  \nabla \eta) u^2 dx  dt, 
			\end{split}
		\end{equation*}
		where the first integral vanishes because of  the Dirichlet boundary condition   $u \in \mathcal{H}_0^1(\Om)$ and  $\nabla \eta=0$ on $\partial \hat{\Om}- \partial \hat{\Om}\cap \partial \Om$.   Since $\eta=0$ outside of $\hat{\Om}$,  
		expanding the divergence   we have  
		\begin{equation*}
			\begin{split}
				\left((P_1 u)_2, (P_2 u)_1\right)_{L^2(Q)}
				=& s^3 \lambda^3 \int_{0}^{T}\int_{\hat{\Om}} \Div(\xi^3  (A \nabla  \eta \cdot \nabla \eta)   A  \nabla \eta) u^2 dx  dt\\
				=&s^3 \lambda^3 \iint_Q \xi^3\Div(  A \nabla  \eta)    (A \nabla  \eta \cdot \nabla \eta)  u^2 dx  dt\\%7
				&\qquad +3s^3 \lambda^4 \iint_Q \left|  A \nabla  \eta \cdot \nabla \eta\right| ^2 \xi^3 u^2 dx  dt\\%8
				&\qquad +s^3 \lambda^3 \iint_Q \xi^3  A \nabla  \eta   \cdot \nabla(  A \nabla  \eta\cdot \nabla  \eta)u^2 dx  dt\\%9
				=&: T_2 + T_3 +T_4.
			\end{split}
		\end{equation*}
		We clearly have that 
		$T_1+T_3$ is a positive term. However, we need to use them to control some other possibly negative terms, so we need to further estimate them. By  the properties of $\eta$, we have
		\begin{equation}\label{T1T3}
			\begin{split}
				T_1+T_3=&s^3 \lambda^4 \iint_Q \left|  A \nabla  \eta \cdot \nabla \eta\right| ^2 \xi^3 u^2 dx  dt\\
				\ge&   s^3 \lambda^4 \iint_Q \left|  A \nabla  \eta \cdot \nabla \eta\right| ^2 \xi^3 u^2 dx  dt \   - s^3 \lambda^4 \int_{0}^{T} \int_{\om}   \xi^3 u^2 dx  dt=: T^*+\hat{T} \,. 
			\end{split}
		\end{equation}
		The term $T_2$ is estimated as follows.  
		\begin{equation*}
			\begin{split}
				T_2=&s^3 \lambda^3 \iint_Q \xi^3\Div(  A \nabla  \eta)    (A \nabla  \eta \cdot \nabla \eta)  u^2 dx  dt\\%1
				=&s^3 \lambda^3 \int_{0}^{T}\int_{\om} \xi^3\Div(  A \nabla  \eta)    (A \nabla  \eta \cdot \nabla \eta)  u^2 dx  dt\\%2
				&\qquad +s^3 \lambda^3 \int_{0}^{T}\int_{\Om\backslash\om} \xi^3\Div(  A \nabla  \eta)    (A \nabla  \eta \cdot \nabla \eta)  u^2 dx  dt\\%3
				=&:T_{21} + T_{22}.
				%			\ge&-Cs^3 \lambda^3 \int_{0}^{T}\int_{\om} \xi^3  u^2 dx  dt -Cs^3 \lambda^3 \int_{0}^{T}\int_{\Om\backslash\om} \xi^3  u^2 dx  dt,\\%4
			\end{split}
		\end{equation*}
		Since $A\in C^2(\Omega)$ and $\eta\in C_0^3(\Omega)$, $\Div(  A \nabla  \eta)   ( A \nabla  \eta \cdot \nabla \eta)$ is bounded in $\Om$,  we have  for some $C=C(\om,\Om)$ 
		\begin{equation*}
			\begin{split}
				T_{21}=s^3 \lambda^3 \int_{0}^{T}\int_{\om} \xi^3\Div(  A \nabla  \eta)    (A \nabla  \eta \cdot \nabla \eta)  u^2 dx  dt
				\ge&-Cs^3 \lambda^3 \int_{0}^{T}\int_{\om} \xi^3  u^2 dx  dt,
			\end{split}
		\end{equation*}
		and
		\begin{equation*}
			\begin{split}
				T_{22}=s^3 \lambda^3 \int_{0}^{T}\int_{\Om\backslash\om} \xi^3\Div(  A \nabla  \eta)    (A \nabla  \eta \cdot \nabla \eta)  u^2 dx  dt
				\ge&-Cs^3 \lambda^3 \int_{0}^{T}\int_{\Om\backslash\om} \xi^3  u^2 dx  dt.
			\end{split}
		\end{equation*}
		From $\left\| \nabla \eta \right\|_{L^2(\Om)} \ge C >0$ in $\Om_0$, and $\Om\backslash\om \subset \Om_0$,  it follows 
		\begin{equation*}
			\begin{split}
				-Cs^3 \lambda^3 \int_{0}^{T}\int_{\Om\backslash\om} \xi^3  u^2 dx  dt\ge&
				-Cs^3 \lambda^3 \int_{0}^{T}\int_{\Om\backslash\om} \xi^3 |  A \nabla \eta \cdot \nabla \eta |^2 u^2 dx  dt\\
				\ge&
				-Cs^3 \lambda^3 \iint_{Q} \xi^3 |  A \nabla \eta \cdot \nabla \eta |^2 u^2 dx  dt.
			\end{split}
		\end{equation*}
		Thus we have
		\begin{equation}\label{T2}
			\begin{split}
				T_2
				\ge-Cs^3 \lambda^3 \int_{0}^{T}\int_{\om} \xi^3  u^2 dx  dt 	-Cs^3 \lambda^3 \iint_{Q} \xi^3 |  A \nabla \eta \cdot \nabla \eta |^2 u^2 dx  dt.
			\end{split}
		\end{equation}
		In a similar way, we can deduce
		\begin{equation}\label{T4}
			\begin{split}
				T_4=&s^3 \lambda^3 \iint_Q \xi^3  A \nabla  \eta   \cdot \nabla(  A \nabla  \eta\cdot \nabla  \eta)u^2 dx  dt\\
				\ge&-Cs^3 \lambda^3 \int_{0}^{T}\int_{\om} \xi^3  u^2 dx  dt 	-Cs^3 \lambda^3 \iint_{Q} \xi^3 |  A \nabla \eta \cdot \nabla \eta |^2 u^2 dx  dt.
			\end{split}
		\end{equation}
		Hence, the terms $T_2$ and $T_4$ are absorbed by $T^*$ and $\hat{T}$ by simply taking $\lambda \ge C$
		(which means that there is $\lambda\ge C$ such that 
		$T_2+T_4+T^*+\hat{T}\ge C (T^*+\hat{T})$ for some positive $C >0$).
		
		By items ($i$) and    ($ii$) for the solution $u$, we have
		\begin{equation*}
			\begin{split}
				T_5:=&	\left((P_1 u)_3, (P_2 u)_1\right)_{L^2(Q)}= 
				s^2 \lambda^2 \iint_Q \xi^2  (A \nabla \eta \cdot \nabla \eta) u u_t dx  dt\\
				=&\frac{1}{2}s^2 \lambda^2 \int_\Om \xi^2  (A \nabla \eta \cdot \nabla \eta) u^2 dx \bigg|_0^T - s^2 \lambda^2 \iint_Q \xi \xi_t  (A \nabla \eta \cdot \nabla \eta) u^2 dx  dt\\
				=&- s^2 \lambda^2 \iint_Q \xi \xi_t  (A \nabla \eta \cdot \nabla \eta) u^2 dx  dt\,. 
			\end{split}
		\end{equation*}
		By the definitions of $\xi$ in \eqref{e.2.7}, we have $\xi\xi_t\le C\xi^3$. Similar to $T_2$,
		\begin{equation}\label{T5}
			\begin{split}
				T_5\ge- Cs^2 \lambda^2 \iint_Q \xi^3 |  A \nabla \eta \cdot \nabla \eta |^2 u^2 dx  dt
				- Cs^2 \lambda^2 \int_{0}^{T}\int_{\om} \xi^3  u^2 dx  dt 
			\end{split}
		\end{equation}
		for some $C=C(\om,\Om,T)$. This term can   be absorbed by $T^*$ and $\hat{T}$. Consequently, by  \eqref{T2},\eqref{T1T3}, \eqref{T4}, and \eqref{T5}, we have
		\begin{equation}\label{P1P21}
			\begin{split}
				\big((P_1 u),& (P_2 u)_1\big) _{L^2(Q)}\\
				=&\left((P_1 u)_1, (P_2 u)_1\right)_{L^2(Q)} + \left((P_1 u)_2, (P_2 u)_1\right)_{L^2(Q)}+\left((P_1 u)_3, (P_2 u)_1\right)_{L^2(Q)}\\
				\ge&C s^3 \lambda^4 \iint_Q \left|  A \nabla  \eta \cdot \nabla \eta\right| ^2 \xi^3 u^2 dx  dt - Cs^3 \lambda^4 \int_{0}^{T} \int_{\om}   \xi^3 u^2 dx  dt.
			\end{split}
		\end{equation}
		On the other hand, we have
		\begin{equation*}
			\begin{split}
				\left((P_1 u)_1, (P_2 u)_2\right)_{L^2(Q)}=&-2s \lambda^2 \iint_Q \xi   (A \nabla \eta \cdot \nabla \eta)  \Div( A \nabla u) u dx  dt\\
				=&-2s \lambda^2 \int_{0}^{T} \int_{\hat{\Om}}\xi   (A \nabla \eta \cdot \nabla \eta)  \Div( A \nabla u) u dx  dt\\
				=&-2s \lambda^2 \int_{0}^{T} \int_{\partial\hat{\Om}} \xi u  (A \nabla \eta \cdot \nabla \eta)  ( A \nabla u \cdot \nu) ds  dt\\
				&\qquad+2s \lambda^2 \int_{0}^{T} \int_{\hat{\Om}} \nabla (u \xi   (A \nabla \eta \cdot \nabla \eta) ) \cdot  A \nabla u dx  dt\\
				=&-2s \lambda^2 \int_{0}^{T} \int_{\partial\hat{\Om}} \xi u  (A \nabla \eta \cdot \nabla \eta)  ( A \nabla u \cdot \nu) ds  dt\\
				&\qquad+2s \lambda^2 \iint_Q \xi  (A \nabla \eta \cdot \nabla \eta)  (A \nabla u \cdot \nabla u)   dx  dt\\
				&\qquad+2s \lambda^2 \iint_Q \xi  u  A \nabla u  \cdot \nabla(  A \nabla \eta \cdot \nabla \eta)dx  dt\\
				&\qquad +2s \lambda^3 \iint_Q \xi u  (A \nabla \eta \cdot \nabla \eta)  (A \nabla u \cdot \nabla \eta)   dx  dt\\
				=&:B_2+T_6+T_7+T_8.
			\end{split}
		\end{equation*}
		Similarly to the previous analysis, we can deduce that $B_2=0$. We will not need to process  $T_6$ since it is positive and we do not plan to use it to control other negative terms.   For $T_7$ and $T_8$, similar to the estimate of $T_2$, 
		we have
		\begin{equation*}
			\begin{split}
				T_8=&2s \lambda^3 \iint_Q \xi u  (A \nabla \eta \cdot \nabla \eta)  (A \nabla u \cdot \nabla \eta)   dx  dt\\
				\ge&-Cs^2 \lambda^4 \iint_Q \xi \left|  A \nabla  \eta \cdot \nabla \eta\right| ^2 u^2  dx  dt\\
				&\qquad -C \lambda^2 \iint_Q \xi   (A \nabla \eta \cdot \nabla \eta)  (A \nabla u \cdot \nabla u) dx  dt 
			\end{split}
		\end{equation*}
		and 
		%\begin{equation*}
		%	\begin{split}
			%		C_1=2s \lambda^2 \iint_Q \xi  A \nabla \eta \cdot \nabla \eta  A \nabla u \cdot \nabla u   dx  dt\ge Cs \lambda^2 \iint_Q \xi   A \nabla u \cdot \nabla u   dx  dt,
			%	\end{split}
		%\end{equation*}
		%and
		\begin{equation*}
			\begin{split}
				T_7=&2s \lambda^2 \iint_Q \xi u\nabla(  A \nabla \eta \cdot \nabla \eta) \cdot  A \nabla u  dx  dt\\
				\ge&-Cs^2 \lambda^3 \iint_Q \xi  A \nabla(  A \nabla \eta \cdot \nabla \eta) \cdot \nabla(  A \nabla \eta \cdot \nabla \eta) u^2  dx  dt\\
				&\qquad -C \lambda \iint_Q \xi  (A \nabla u \cdot \nabla u) dx  dt\\
				\ge& -Cs^2 \lambda^3 \iint_Q \left|  A \nabla  \eta \cdot \nabla \eta\right| ^2 \xi u^2 dx  dt - Cs^2 \lambda^3 \int_{0}^{T} \int_{\om}   \xi u^2 dx  dt\\
				&\qquad -C \lambda \iint_Q \xi  (A \nabla \eta \cdot \nabla \eta) (A \nabla u \cdot \nabla u) dx  dt
				-C \lambda \int_{0}^{T}\int_\om \xi  A \nabla u \cdot \nabla u dx  dt.
			\end{split}
		\end{equation*}
		
		Thus, we have 
		\begin{equation}\label{P11P22}
			\begin{split}
				\big((P_1 u)_1,& (P_2 u)_2\big)_{L^2(Q)}\\
				\ge&T_6
				-Cs^2 \lambda^4 \iint_Q \left|  A \nabla  \eta \cdot \nabla \eta\right| ^2 \xi u^2 dx  dt - Cs^2 \lambda^3 \int_{0}^{T} \int_{\om}   \xi u^2 dx  dt\\
				&\quad -C \lambda^2 \iint_Q \xi   (A \nabla \eta \cdot \nabla \eta)  (A \nabla u \cdot \nabla u) dx  dt
				-C \lambda \int_{0}^{T}\int_\om \xi  A \nabla u \cdot \nabla u dx  dt.
			\end{split}
		\end{equation}
		We also have
		\begin{equation*}
			\begin{split}
				\left((P_1 u)_2, (P_2 u)_2\right)_{L^2(Q)}
				=& -2s \lambda \iint_Q \xi   (A \nabla  u \cdot \nabla \eta) 
				\Div( A \nabla u) dx dt\\
				=& -2s \lambda \int_{0}^{T}\int_{\hat{\Om}} \xi   (A \nabla  u \cdot \nabla \eta) 
				\Div( A \nabla u) dx dt\\
				=& - 2s \lambda \int_{0}^{T}\int_{\partial\hat{\Om}} \xi \left(  A\nabla u\cdot \nabla \eta \right) (A \nabla u \cdot \nu) ds dt\\
				&\qquad+2s \lambda^2 \iint_Q \xi   (A \nabla  u \cdot \nabla \eta)   (A \nabla  u \cdot \nabla \eta) dx dt\\
				&\qquad +2s \lambda \iint_Q \xi 
				A \nabla u \cdot \nabla(   A \nabla  u \cdot \nabla \eta) dx dt.
			\end{split}
		\end{equation*}
		Denote $- 2s \lambda \int_{0}^{T}\int_{\partial\hat{\Om}} \xi \left(  A\nabla u\cdot \nabla \eta \right) (A \nabla u \cdot \nu) ds dt$ by $B_3$, which is a boundary term  and denote $2s \lambda^2 \iint_Q \xi   (A \nabla  u \cdot \nabla \eta)   (A \nabla  u \cdot \nabla \eta) dx dt$, which  is a positive term.
		Furthermore,
		\begin{equation*}
			\begin{split}
				&2s \lambda \iint_Q \xi   A \nabla u \cdot \nabla(   A \nabla  u \cdot \nabla \eta)dx dt\\
				=&2s \lambda \sum_{i=1}^{N}\iint_Q  \xi \left(  A \nabla  \eta\right)_i  A \nabla  u \cdot \frac{\partial}{\partial x_i} (\nabla u)  + \xi(\nabla u)_i   A \nabla  u \cdot \nabla \left(  A \nabla  \eta\right) _i  dx dt.\\
			\end{split}
		\end{equation*}
		We denote $2s \lambda \sum_{i=1}^{N}\iint_Q \xi (\nabla u)_i   A \nabla  u \cdot \nabla \left(  A \nabla  \eta\right) _i  dx dt$ by $T_{10}$, the $i$ here refers to the i-th component. After some additional computations we   see that
		\begin{equation*}
			\begin{split}
				&2s \lambda \sum_{i=1}^{N}\iint_Q  \xi \left(  A \nabla  \eta\right)_i  A \nabla  u \cdot \frac{\partial}{\partial x_i} (\nabla u) dx dt\\
				=& s \lambda \sum_{i=1}^{N}\iint_Q  \xi \left(  A \nabla  \eta\right)_i \left[ \frac{\partial}{\partial x_i} \left( A \nabla  u \cdot \nabla u  \right) - \frac{\partial( A)}{\partial x_i} \nabla  u \cdot \nabla u \right]  dx dt\\
				=&s \lambda \int_{0}^{T}\int_{\partial\hat{\Om}} \xi \left(  A\nabla u\cdot \nabla u \right) (A\nabla \eta \cdot \nu) ds dt
				-s \lambda \iint_Q \xi   (A \nabla u \cdot \nabla u) \Div( A \nabla \eta)  dx dt\\
				&\qquad -s \lambda \sum_{i=1}^{N}\iint_Q \xi \left(  A \nabla  \eta\right)_i \frac{\partial( A)}{\partial x_i} \nabla  u \cdot \nabla u   dx dt - s\lambda^2 \iint_Q \xi   (A \nabla \eta \cdot \nabla \eta ) (A \nabla u \cdot \nabla u) dx  dt\\
				=&:B_{4}+T_{11}+T_{12} +T_{13}.
			\end{split}
		\end{equation*}
		It is easy to see that 
		\begin{equation}\label{T6T13}
			T_6 + T_{13}=s\lambda^2 \iint_Q \xi   (A \nabla \eta \cdot \nabla \eta)  (A \nabla u \cdot \nabla u) dx  dt\ge 0.
		\end{equation}
		Since $u\in \mathcal{H}_0^1(\Omega)$, by virtue of the properties satisfied by $\eta$ in \eqref{ETA1} and \eqref{ETA2}, and that   $\nabla u,\nabla \eta$ are parallel to the outward unit normal vector $\nu$ on $\partial\hat{\Om}\cap\partial\Om$ ($\nabla u = \pm|\nabla u|\nu$, $\nabla \eta = -|\nabla \eta|\nu$), $\nabla \eta =0$ on $\partial \hat{\Om}- \partial \hat{\Om}\cap \partial \Om$,   we have 
		\begin{equation*}
			\begin{split}
				B_3 + B_{4}=& - 2s \lambda \int_{0}^{T}\int_{\partial\hat{\Om}} \xi \left(  A\nabla u\cdot \nabla \eta \right) A \nabla u \cdot \nu ds dt\\
				&\qquad + 
				s \lambda \int_{0}^{T}\int_{\partial\hat{\Om}} \xi \left(  A\nabla u\cdot \nabla u \right) A\nabla \eta \cdot \nu ds dt\\%1
				=& - 2s \lambda \int_{0}^{T}\int_{\partial \hat{\Om}\cap \partial \Om} \xi \left(  A\nabla u\cdot \nabla \eta \right) A \nabla u \cdot \nu ds dt\\ %2
				&\qquad  - 2s \lambda \int_{0}^{T}\int_{\partial \hat{\Om}- \partial \hat{\Om}\cap \partial \Om} \xi \left(  A\nabla u\cdot \nabla \eta \right) A \nabla u \cdot \nu ds dt  \\%3
				&\qquad + s \lambda \int_{0}^{T}\int_{\partial \hat{\Om}\cap \partial \Om} \xi \left(  A\nabla u\cdot \nabla u \right) A\nabla \eta \cdot \nu ds dt\\%4
				&\qquad   +s \lambda \int_{0}^{T}\int_{\partial \hat{\Om}- \partial \hat{\Om}\cap \partial \Om} \xi \left(  A\nabla u\cdot \nabla u \right) A\nabla \eta \cdot \nu ds dt \\%5
				=& - 2s \lambda \int_{0}^{T}\int_{\partial \hat{\Om}\cap \partial \Om} \xi \left(  A\nabla u\cdot \nabla \eta \right) A \nabla u \cdot \nu ds dt\\%6
				&\qquad + s \lambda \int_{0}^{T}\int_{\partial \hat{\Om}\cap \partial \Om} \xi \left(  A\nabla u\cdot \nabla u \right) A\nabla \eta \cdot \nu ds dt\\%7
				=& s \lambda \int_{0}^{T}\int_{\partial \hat{\Om}\cap \partial \Om} \xi |\nabla u|^2 |\nabla \eta| |A\nu \cdot \nu|^2 ds dt 
				\ge 0.
			\end{split}
		\end{equation*}
		For $T_{10}$, we have
		\begin{equation*}
			\begin{split}
				T_{10}=&2s \lambda \sum_{i=1}^{N}\iint_Q \xi (\nabla u)_i   A \nabla  u \cdot \nabla \left(  A \nabla  \eta\right) _i  dx dt\\
				=&2s \lambda \sum_{i=1}^{N} \iint_Q (\xi^{\frac{1}{2}}   A^{\frac{1}{2}} \nabla u)(\xi^{\frac{1}{2}}   A^{\frac{1}{2}} \nabla \left(  A \nabla  \eta\right) _i  (\nabla u)_i ) dx dt\\
				\ge& -Cs \lambda \iint_Q \xi   A\nabla u  \cdot \nabla u dx dt\\
				&\qquad  -Cs \lambda \sum_{i=1}^{N}\iint_Q \xi   A\nabla \left(  A \nabla  \eta\right) _i \cdot \nabla \left(  A \nabla  \eta\right) _i  (\nabla u)_i \cdot (\nabla u)_i dx dt\\
				\ge& -Cs \lambda \iint_Q \xi   A\nabla u  \cdot \nabla u dx dt\\
				&\qquad  -Cs \lambda \sum_{i=1}^{N}\int_{0}^{T}\int_{\hat{\Om}} \xi   A\nabla \left(  A \nabla  \eta\right) _i \cdot \nabla \left(  A \nabla  \eta\right) _i  (\nabla u)_i \cdot (\nabla u)_i dx dt.
			\end{split}
		\end{equation*}
		Since $A\nabla \left(  A \nabla  \eta\right) _i \cdot \nabla \left(  A \nabla  \eta\right) _i \in C^0(\overline{\Omega})$, similar to $T_2$, we have
		\begin{equation*}
			\begin{split}
				-Cs \lambda \sum_{i=1}^{N}&\int_{0}^{T}\int_{\hat{\Om}} \xi   A\nabla \left(  A \nabla  \eta\right) _i \cdot \nabla \left(  A \nabla  \eta\right) _i  (\nabla u)_i \cdot (\nabla u)_i dx dt\\
				\ge&
				-Cs \lambda \sum_{i=1}^{N}\int_{0}^{T}\int_{\hat{\Om}} \xi     (\nabla u)_i \cdot (\nabla u)_i dx dt\\
				\ge&
				-Cs \lambda \int_{0}^{T}\int_{\hat{\Om}} \xi   A\nabla u  \cdot \nabla u dx dt\\
				\ge&-Cs \lambda \iint_Q \xi   A\nabla u  \cdot \nabla u dx dt\\
				\ge&-C s\lambda \iint_Q \xi   A \nabla \eta \cdot \nabla \eta  A \nabla u \cdot \nabla u dx  dt -Cs \lambda \int_{0}^{T}\int_\om \xi   A\nabla u  \cdot \nabla u dx dt.
			\end{split}
		\end{equation*} 
		In a similar way, we also have 
		\begin{equation*}
			\begin{split}
				T_{11}=&-s \lambda \iint_Q \xi   A \nabla u \cdot \nabla u \Div( A \nabla \eta)  dx dt\\
				\ge&  -Cs \lambda \iint_Q \xi   A\nabla u  \cdot \nabla u dx dt \\
				\ge&-C s\lambda \iint_Q \xi   A \nabla \eta \cdot \nabla \eta  A \nabla u \cdot \nabla u dx  dt -Cs \lambda \int_{0}^{T}\int_\om \xi   A\nabla u  \cdot \nabla u dx dt,
			\end{split}
		\end{equation*}  
		and
		\begin{equation*}
			\begin{split}
				T_{12}=&-s \lambda \sum_{i=1}^{N}\iint_Q \xi \left(  A \nabla  \eta\right)_i \frac{\partial( A)}{\partial x_i} \nabla  u \cdot \nabla u   dx dt\\
				\ge&-Cs \lambda \iint_Q \xi   A\nabla u  \cdot \nabla u dx dt\\
				\ge&  -C s\lambda \iint_Q \xi   A \nabla \eta \cdot \nabla \eta  A \nabla u \cdot \nabla u dx  dt -Cs \lambda \int_{0}^{T}\int_\om \xi   A\nabla u  \cdot \nabla u dx dt. 
			\end{split}
		\end{equation*}
		Consequently,
		\begin{equation}\label{P12P22}
			\begin{split}
				\left((P_1 u)_2, (P_2 u)_2\right)_{L^2(Q)}
				\ge&
				T_{13} + Cs \lambda^2 \iint_Q \xi  \left|  A \nabla  u \cdot \nabla \eta\right| ^2  dx  dt\\
				&\qquad -C s\lambda \iint_Q \xi   A \nabla \eta \cdot \nabla \eta  A \nabla u \cdot \nabla u dx  dt\\
				&\qquad -Cs \lambda \int_{0}^{T}\int_\om \xi   A\nabla u  \cdot \nabla u dx dt\,. 
			\end{split}
		\end{equation}
		Since $u_t \in \mathcal{H}_0^1(\Om)$ and $\nabla \eta=0$ on $\partial \hat{\Om}- \partial \hat{\Om}\cap \partial \Om$, we have 
		\[
		B_5 :=\iint_\Sigma u_t  (A\nabla u \cdot \nu) ds dt=0\,.
		\]
		Thus, we find that
		\begin{equation}\label{P13P22}
			\begin{split}
				\left((P_1 u)_3, (P_2 u)_2\right)_{L^2(Q)}
				=
				& \iint_Q \Div( A \nabla  u) u_t dx  dt\\
				=&\iint_\Sigma u_t  (A\nabla u \cdot \nu) ds dt - \iint_{Q}  A\nabla u \cdot \nabla u_t dxdt\\
				=&  - \frac12 \iint_{Q} \frac{d}{dt}  \left(A\nabla u \cdot \nabla u \right)  dxdt\\
				=&-\frac{1}{2}\int_\Om  A\nabla u \cdot \nabla u dx \bigg|_0^T =0 
			\end{split}
		\end{equation}
		%We denote $\iint_\Sigma u_t  (A\nabla u \cdot \nu) ds dt$ by $B_{5}$.  
		Now from \eqref{P11P22}, \eqref{T6T13}, \eqref{P12P22}, and \eqref{P13P22},it follows that for sufficiently large $\lambda$,  
		\begin{equation}\label{P1P22}
			\begin{split}
				\big((P_1 u), &(P_2 u)_2\big)_{L^2(Q)}\\
				=&\left((P_1 u)_1, (P_2 u)_2\right)_{L^2(Q)} + \left((P_1 u)_2, (P_2 u)_2\right)_{L^2(Q)}+\left((P_1 u)_3, (P_2 u)_2\right)_{L^2(Q)}\\
				\ge&Cs \lambda^2 \iint_Q \xi  \left|  A \nabla  u \cdot \nabla \eta\right| ^2  dx  dt
				+C s\lambda^2 \iint_Q \xi   A \nabla \eta \cdot \nabla \eta  A \nabla u \cdot \nabla u dx  dt\\
				& \qquad  -Cs^2 \lambda^4 \iint_Q \left|  A \nabla  \eta \cdot \nabla \eta\right| ^2 \xi u^2 dx  dt
				- Cs^2 \lambda^3 \int_{0}^{T} \int_{\om}   \xi u^2 dx  dt\\
				&\qquad  -Cs \lambda \int_{0}^{T}\int_\om \xi   A\nabla u  \cdot \nabla u dx dt.
			\end{split}
		\end{equation}
		Let us now consider the scalar product
		$\left((P_1 u)_1, (P_2 u)_3\right)_{L^2(Q)}$,
		\begin{equation*}
			\begin{split}
				\left((P_1 u)_1, (P_2 u)_3\right)_{L^2(Q)}
				=&-2s^2\lambda^2 \iint_Q \xi \sigma_t  A \nabla  \eta \cdot \nabla \eta u^2 dxdt=:T_{14}\,. 
			\end{split}
		\end{equation*}
		Since $\sigma_t\xi \le C\xi^3$, we have
		\begin{equation*}
			\begin{split}
				T_{14}\ge&-Cs^2\lambda^2  \iint_Q \xi^3 u^2 dxdt\\
				\ge&-Cs^2\lambda^2  \int_0^T\int_{\omega} \xi^3 u^2 dxdt
				-Cs^2 \lambda^2 \iint_Q \left|  A \nabla  \eta \cdot \nabla \eta\right| ^2 \xi^3 u^2 dx  dt.
			\end{split}
		\end{equation*}
		Obviously, this is absorbed by $T^*$ and $\hat{T}$ if we choose $s$ and $\lambda$   sufficiently large.
		Now we consider 
		\begin{equation*}
			\begin{split}
				\left((P_1 u)_2, (P_2 u)_3\right)_{L^2(Q)}
				=&-2s^2\lambda \iint_Q \xi \sigma_t  A \nabla  \eta \cdot \nabla  u u dxdt\\%1
				=&-s^2\lambda \iint_Q \xi \sigma_t  A \nabla  \eta \cdot \nabla  (u^2) dxdt\\%2
				=&-s^2\lambda \int_{0}^{T}\int_{\hat{\Om}} \xi \sigma_t  A \nabla  \eta \cdot \nabla  (u^2) dxdt\\%3
				=&-s^2\lambda \int_{0}^{T}\int_{\partial\hat{\Om}} u^2 \xi  \sigma_t  (A \nabla  \eta \cdot \nu) dsdt\\%4
				&\qquad  +s^2\lambda \int_{0}^{T}\int_{\hat{\Om}} u^2 \Div(\xi  \sigma_t  A \nabla  \eta) dxdt\\
				=&-s^2\lambda \int_{0}^{T}\int_{\partial\hat{\Om}} u^2 \xi  \sigma_t  (A \nabla  \eta \cdot \nu) dsdt+ s^2\lambda^2 \iint_Q \xi \sigma_t  A \nabla  \eta \cdot \nabla  \eta u^2 dxdt\\
				&\qquad+s^2\lambda \iint_Q \xi   A \nabla  \eta  \cdot \nabla \sigma_t u^2 dxdt
				+s^2\lambda \iint_Q \xi  \sigma_t \Div( A \nabla  \eta)  u^2 dxdt\\
				=&:B_6 +T_{15} +T_{16} +T_{17}.
			\end{split}
		\end{equation*}
		Using the same argument as that in the previous proof, we have $B_6 =0$ and
		\begin{equation*}
			\begin{split}
				T_{15}
				=& s^2\lambda^2 \iint_Q \xi \sigma_t  A \nabla  \eta \cdot \nabla  \eta u^2 dxdt\\
				\ge&-Cs^2\lambda^2 \int_{0}^{T}\int_\om \xi^3 u^2 dxdt
				-Cs^2\lambda^2 \iint_Q \xi^3 \left|  A \nabla  \eta \cdot \nabla \eta\right| ^2 u^2 dxdt.
			\end{split}
		\end{equation*}
		For $T_{16}$ and $T_{17}$ we have,
		\begin{equation*}
			\begin{split}
				T_{16}
				=&s^2\lambda \iint_Q \xi \nabla \sigma_t  A \nabla  \eta  u^2 dxdt\\
				\ge&-Cs^2\lambda^2 \int_{0}^{T}\int_\om \xi^3 u^2 dxdt
				-Cs^2\lambda^2 \iint_Q \xi^3 \left|  A \nabla  \eta \cdot \nabla \eta\right| ^2 u^2 dxdt,
			\end{split}
		\end{equation*}
		and
		\begin{equation*}
			\begin{split}
				T_{17}
				=&s^2\lambda \iint_Q \xi  \sigma_t \Div( A \nabla  \eta)  u^2 dxdt\\
				\ge&-Cs^2\lambda \int_{0}^{T}\int_\om \xi^3 u^2 dxdt
				-Cs^2\lambda \iint_Q \xi^3 \left|  A \nabla  \eta \cdot \nabla \eta\right| ^2 u^2 dxdt.
			\end{split}
		\end{equation*}
		Thus, we have
		\begin{equation*}
			\begin{split}
				\left((P_1 u)_2, (P_2 u)_3\right)_{L^2(Q)}
				\ge&-Cs^2\lambda^2 \int_{0}^{T}\int_\om \xi^3 u^2 dxdt
				-Cs^2\lambda^2 \iint_Q \xi^3 \left|  A \nabla  \eta \cdot \nabla \eta\right| ^2 u^2 dxdt.
			\end{split}
		\end{equation*}
		Denote 
		\begin{equation*}
			\begin{split}
				T:=&\left((P_1 u)_3, (P_2 u)_3\right)_{L^2(Q)}
				= s\iint_{Q}\sigma_t u_t u dx dt = -\frac{1}{2} s \iint_{Q}\sigma_{tt} u^2 dx dt\,. 
			\end{split}
		\end{equation*}
		Then 
		\begin{equation*}
			\begin{split}
				T_{18}\ge&-C s \iint_{Q}\xi^3 u^2 dx dt\\
				\ge&-Cs \int_{0}^{T}\int_\om \xi^3 u^2 dxdt
				-Cs \iint_Q \xi^3 \left|  A \nabla  \eta \cdot \nabla \eta\right| ^2 u^2 dxdt,
			\end{split}
		\end{equation*}
		since $\sigma_{t t}\le C\xi^3$.
		
		Now, we can deduce   that for sufficiently large $\lambda$ and $s$
		\begin{equation}\label{P1P23}
			\begin{split}
				\left((P_1 u), (P_2 u)_3\right)_{L^2(Q)}=&\left((P_1 u)_1, (P_2 u)_3\right)_{L^2(Q)} + \left((P_1 u)_2, (P_2 u)_3\right)_{L^2(Q)}+\left((P_1 u)_3, (P_2 u)_3\right)_{L^2(Q)}\\
				\ge&-Cs^2\lambda^3\int_{0}^{T}\int_\om \xi^3 u^2 dxdt
				-Cs^2\lambda^3\iint_Q \xi^3 \left|  A \nabla  \eta \cdot \nabla \eta\right| ^2 u^2 dxdt.
			\end{split}
		\end{equation}
		Taking into account \eqref{P1P21}, \eqref{P1P22}, and \eqref{P1P23} yields 
		\begin{equation*}
			\begin{split}
				\left( P_1 u, P_2 u \right)_{L^2(Q)}
				\ge&Cs^3\lambda^4\iint_{Q} \xi^3 \left|  A \nabla  \eta \cdot \nabla \eta\right| ^2 u^2 dxdt
				+Cs \lambda^2 \iint_Q \xi  \left|  A \nabla  u \cdot \nabla \eta\right| ^2  dx  dt\\
				&\qquad +C s\lambda^2 \iint_Q \xi   A \nabla \eta \cdot \nabla \eta  A \nabla u \cdot \nabla u dx  dt
				-Cs^3\lambda^4\int_{0}^{T}\int_\om \xi^3 u^2 dxdt\\
				&\qquad  -Cs \lambda \int_{0}^{T}\int_\om \xi   A\nabla u  \cdot \nabla u dx dt.
			\end{split}
		\end{equation*}
		From \eqref{PPP}, for sufficiently large $s$ and $\lambda$  it follows 
		\begin{equation*}
			\begin{split}
				&\left\| P_1 u \right\| ^2_{L^2(Q)} + \left\| P_2 u \right\| ^2_{L^2(Q)}
				+Cs^3\lambda^4\iint_Q \xi^3 \left|  A \nabla  \eta \cdot \nabla \eta\right| ^2 u^2 dxdt\\%1
				&\qquad +Cs \lambda^2 \iint_Q \xi  \left|  A \nabla  u \cdot \nabla \eta\right| ^2  dx  dt
				+C s\lambda^2 \iint_Q \xi   A \nabla \eta \cdot \nabla \eta  A \nabla u \cdot \nabla u dx  dt\\%2
				\le& \left\| f_{s,\lambda} \right\| ^2_{L^2(Q)}+ Cs^3\lambda^4\int_{0}^{T}\int_\om \xi^3 u^2 dxdt	
				+Cs \lambda \int_{0}^{T}\int_\om \xi   A\nabla u  \cdot \nabla u dx dt\\%3
				\le& \left\| e^{-s\sigma} f \right\| ^2_{L^2(Q)}
				+ s^2\lambda^4\iint_{Q} \xi^3 \left|  A \nabla  \eta \cdot \nabla \eta\right| ^2 u^2 dxdt
				+s^2\lambda^2\iint_{Q} \xi^3 \left| \Div( A \nabla  \eta)\right| ^2 u^2 dxdt\\
				&\qquad  + Cs^3\lambda^4\int_{0}^{T}\int_\om \xi^3 u^2 dxdt	
				+Cs \lambda \int_{0}^{T}\int_\om \xi   A\nabla u  \cdot \nabla u dx dt.
			\end{split}
		\end{equation*}
		Thus, we   have
		\begin{equation}\label{P1P2}
			\begin{split}
				\left\| P_1 u \right\| ^2_{L^2(Q)} +& \left\| P_2 u \right\| ^2_{L^2(Q)}
				+Cs^3\lambda^4\iint_Q \xi^3 \left|  A \nabla  \eta \cdot \nabla \eta\right| ^2 u^2 dxdt\\%1
				&\quad  +Cs \lambda^2 \iint_Q \xi  \left|  A \nabla  u \cdot \nabla \eta\right| ^2  dx  dt
				+C s\lambda^2 \iint_Q \xi   A \nabla \eta \cdot \nabla \eta  A \nabla u \cdot \nabla u dx  dt\\%2
				\le& \left\| e^{-s\sigma} f \right\| ^2_{L^2(Q)}
				+ Cs^3\lambda^4\int_{0}^{T}\int_\om \xi^3 u^2 dxdt	
				+Cs \lambda \int_{0}^{T}\int_\om \xi   A\nabla u  \cdot \nabla u dx dt.
			\end{split}
		\end{equation}
		The final step will be to add integrals of $\left|  \Div( A \nabla  u)\right|  ^2$ and $\left|  u_t\right|  ^2$ on  the left-hand side of \eqref{P1P2}. This can be made using the expressions
		of $P_1 u$ and $P_2 u$. Indeed, from item $(iii)$ we have
		\begin{equation*}
			\begin{split}
				&s^{-1} \iint_Q \xi^{-1}\left|u_t\right|^2dx  d t\\
				&\qquad=s^{-1} \iint_Q \xi^{-1}(P_1 u + 2s \lambda^2 \xi u  A \nabla  \eta \cdot \nabla \eta   + 2s \lambda \xi   A \nabla  u \cdot \nabla \eta)^2dx  d t\\
				&\qquad\le  Cs^{-1}\left\|P_1 u\right\|^2+ C s\lambda^4\iint_Q \xi \left|  A \nabla  \eta \cdot \nabla \eta\right| ^2 u^2 dxdt
				+C s\lambda^2\iint_Q \xi \left|  A \nabla  u \cdot \nabla \eta\right| ^2 dxdt,
			\end{split}
		\end{equation*}
		and
		\begin{equation*}
			\begin{split}
				&s^{-1}   \iint_Q \xi^{-1}\left|\Div( A\nabla u)\right|^2dx  d t\\
				&\qquad =  s^{-1}   \iint_Q \xi^{-1}(P_2 u -s^2\lambda^2\xi^2 u  A \nabla \eta \cdot  \nabla \eta - s \sigma_t u)^2dx  d t\\
				&\qquad  \le  C s^{-1} \left\|P_2 u\right\|^2+C s^3 \lambda^4   \iint_Q \xi^3|  A \nabla \eta\cdot \nabla \eta |^2|u|^2dx  d t +C sT^2   \iint_Q \xi^{3}|u|^2dx  d t.
			\end{split}
		\end{equation*}
		Combining with  \eqref{P1P2} yields 
		\begin{equation}\label{P1UP2U}
			\begin{split}
				Cs^{-1} &\iint_Q \xi^{-1}\left|u_t\right|^2dx  d t +Cs^{-1}   \iint_Q \xi^{-1}\left|\Div( A\nabla u)\right|^2dx  d t\\
				&\qquad +Cs^3\lambda^4\iint_Q \xi^3 \left|  A \nabla  \eta \cdot \nabla \eta\right| ^2 u^2 dxdt
				+Cs \lambda^2 \iint_Q \xi  \left|  A \nabla  u \cdot \nabla \eta\right| ^2  dx  dt\\
				&\qquad+C s\lambda^2 \iint_Q \xi  A \nabla \eta \cdot \nabla \eta  A \nabla u \cdot \nabla u dx  dt\\
				\le& \left\| e^{-s\sigma} f \right\| ^2_{L^2(Q)}
				+ Cs^3\lambda^4\int_{0}^{T}\int_\om \xi^3 u^2 dxdt	
				+Cs \lambda \int_{0}^{T}\int_\om \xi   A\nabla u  \cdot \nabla u dx dt,
			\end{split}
		\end{equation}
		for sufficiently large $s$ and $\lambda$.
		It is worth noting that 
		\begin{equation*}
			\begin{split}
				C s\lambda^2 \iint_Q &\xi   A \nabla \eta \cdot \nabla \eta  A \nabla u \cdot \nabla u dx  dt\\%1
				=&C s\lambda^2 \int_0^T \int_{\Om\backslash\om} \xi   A \nabla \eta \cdot \nabla \eta  A \nabla u \cdot \nabla u dx  dt\\
				&\qquad+ C s\lambda^2 \int_0^T \int_{\om} \xi   A \nabla \eta \cdot \nabla \eta  A \nabla u \cdot \nabla u dx  dt\\
				\ge& C s\lambda^2 \int_0^T \int_{\Om\backslash\om} \xi    A \nabla u \cdot \nabla u dx  dt + C s\lambda^2 \int_0^T \int_{\om} \xi   A \nabla \eta \cdot \nabla \eta  A \nabla u \cdot \nabla u dx  dt\\
				\ge&C s\lambda^2 \iint_Q \xi    A \nabla u \cdot \nabla u dx  dt - C s\lambda^2 \int_0^T \int_{\om} \xi   A \nabla u \cdot \nabla u dx  dt,
			\end{split}
		\end{equation*}
		and
		\begin{equation*}
			\begin{split}
				Cs^3\lambda^4\iint_Q &\xi^3 \left|  A \nabla  \eta \cdot \nabla \eta\right| ^2 u^2 dxdt\\
				\ge&Cs^3\lambda^4\iint_Q \xi^3  u^2 dxdt -Cs^3\lambda^4\int_0^T \int_{\om} \xi^3  u^2 dxdt.
			\end{split}
		\end{equation*}
		Hence, \eqref{P1UP2U}  can be written as
		\begin{equation}
			\begin{split}
				Cs^{-1} \iint_Q& \xi^{-1}\left|u_t\right|^2dx  d t +Cs^{-1}   \iint_Q \xi^{-1}\left|\Div( A\nabla u)\right|^2dx  d t\\
				&\qquad+Cs^3\lambda^4\iint_Q \xi^3  u^2 dxdt
				+Cs \lambda^2 \iint_Q \xi  \left|  A \nabla  u \cdot \nabla \eta\right| ^2  dx  dt\\
				&\qquad+C s\lambda^2 \iint_Q \xi   A \nabla u \cdot \nabla u dx  dt\\
				\le& \left\| e^{-s\sigma} f \right\| ^2_{L^2(Q)}
				+ Cs^3\lambda^4\int_{0}^{T}\int_\om \xi^3 u^2 dxdt	
				+Cs \lambda^2 \int_{0}^{T}\int_\om \xi   A\nabla u  \cdot \nabla u dx dt\,. 
			\end{split}
		\end{equation}
		We are now going to  bound the second integral on the right-hand side. To this end, let us introduce a function $\Phi=\Phi(x)$, with
		\begin{equation*}
			\Phi\in C_0^2(\omega_0), \quad 0\le \Phi\le 1,\quad \Phi\equiv 1 \mbox{ in } \omega, \quad \Phi\equiv 0 \mbox{ in } \Om\backslash\om_0, 
		\end{equation*}
		and let us make some computations. We have
		\begin{equation*}
			\begin{split}
				s \lambda^2 \int_{0}^{T}\int_\om \xi  A \nabla u \cdot \nabla udx  d t\le s \lambda^2 \int_{0}^{T}\int_{\om_0} \Phi \xi  A \nabla u \cdot \nabla udx  d t,
			\end{split}
		\end{equation*}
		and
		\begin{equation*}
			\begin{split}
				s \lambda^2 \int_{0}^{T}&\int_{\om_0}\Phi \xi  A \nabla u \cdot \nabla udx  d t\\%1
				=&- s \lambda^2 \int_{0}^{T}\int_{\om_0} \Div(\Phi\xi  A\nabla u) udx  d t\\%2
				=&- s \lambda^3 \int_{0}^{T}\int_{\om_0} \Phi\xi u  A\nabla u \cdot \nabla \eta dx  d t - s \lambda^2 \int_{0}^{T}\int_{\om_0}\Phi \xi\Div(  A\nabla u) udx  d t\\%3
				&\qquad- s \lambda^2 \int_{0}^{T}\int_{\om_0}  \xi  A\nabla u\cdot\nabla\Phi udx  d t
				\\%4
				\le& C \lambda^2 \int_{0}^{T}\int_{\om_0} \xi \left|  A \nabla  u \cdot \nabla \eta\right| ^2 dx  d t + Cs^2 \lambda^4 \int_{0}^{T}\int_{\om_0} \xi u^2dx  d t\\
				&\qquad+C \epsilon s^{-1}  \int_{0}^{T}\int_{\om_0} \xi^{-1} \left| \Div(  A\nabla u)\right| ^2dx  d t + C \frac{1}{\epsilon}s^3 \lambda^4 \int_{0}^{T}\int_{\om_0} \xi^3 u^2dx  d t\\
				&\qquad+C \int_{0}^{T}\int_{\om_0}   A \nabla u \cdot \nabla u  dx  d t + Cs^2 \lambda^4 \int_{0}^{T}\int_{\om_0}  \xi^2 u^2 dx  d t.
			\end{split}
		\end{equation*}
		Thus, for $\epsilon$ that is small enough, the term $s \lambda^2 \int_{0}^{T}\int_{\om} \xi  A \nabla u \cdot \nabla udx d t$ can be absorbed by the other terms. 
		
		This implies that for sufficiently large $s$ and $\lambda$, we have
		\begin{equation*}
			\begin{split}
				Cs^{-1} \iint_Q & \xi^{-1}\left|u_t\right|^2dx  d t +Cs^{-1}   \iint_Q \xi^{-1}\left|\Div( A\nabla u)\right|^2dx  d t\\
				&\qquad+Cs^3\lambda^4\iint_Q \xi^3  u^2 dxdt
				+Cs \lambda^2 \iint_Q \xi  \left|  A \nabla  u \cdot \nabla \eta\right| ^2  dx  dt\\
				&\qquad+C s\lambda^2 \iint_Q \xi    A \nabla u \cdot \nabla u dx  dt\\
				\le& \left\| e^{-s\sigma} f \right\| ^2_{L^2(Q)}
				+ Cs^3\lambda^4\int_{0}^{T}\int_{\om_0} \xi^3 u^2 dxdt.
			\end{split}
		\end{equation*}
		By replacing $u$ with $e^{-s\sigma}w$, we can then revert back to the original variable $w$ and conclude the result.
		\begin{equation*}
			\begin{split}
				Cs^{-1} \iint_Q & e^{-2s\sigma}\xi^{-1}\left|w_t\right|^2dx  d t +Cs^{-1}   \iint_Q e^{-2s\sigma}\xi^{-1}\left|\Div( A\nabla w)\right|^2dx  d t\\
				&\qquad+Cs^3\lambda^4\iint_Q e^{-2s\sigma}\xi^3  w^2 dxdt
				+Cs \lambda^2 \iint_Q e^{-2s\sigma}\xi  \left|  A \nabla  w \cdot \nabla \eta\right| ^2  dx  dt\\
				&\qquad+C s\lambda^2 \iint_Q e^{-2s\sigma}\xi    A \nabla w \cdot \nabla w dx  dt\\
				\le& \left\| e^{-s\sigma} f \right\| ^2_{L^2(Q)}
				+ Cs^3\lambda^4\int_{0}^{T}\int_{\om_0} e^{-2s\sigma}\xi^3 w^2 dxdt.
			\end{split}
		\end{equation*}
		This completes the proof of Theorem \ref{TH2}.  
	\end{proof}
	
	Next, we turn to prove Carleman estimate for Neumann boundary condition.  It is more complex. 
	
	\begin{proof}[Proof of Theorem \ref{TH22}.]  %({\color{blue} The Detailed calculation in Chapter 5})]	
		Recall  
		\begin{equation*}
			\begin{split}
				& \theta(t):=[t(T-t)]^{-4}, \quad \xi(x, t):=\theta(t) e^{ \lambda(8|\eta|_\infty+\eta (x))}, \quad \sigma(x, t):=\theta(t) e^{10 \lambda|\eta|_\infty}-\xi(x, t).
			\end{split}
		\end{equation*}
		For the same $\eta$ defined above, consider 
		\begin{equation*}
			\begin{split}
				& \widetilde{\xi}(x, t):=\theta(t) e^{ \lambda(8|\eta|_\infty-\eta (x))}, \quad \wt{\sigma}(x, t):=\theta(t) e^{10 \lambda|\eta|_\infty}-\wt{\xi}(x, t) 
			\end{split}
		\end{equation*}
		and %	Similarly, let $s>s_0>0$ and let
		\begin{equation}\label{Defineu}
			\wt{u}(x,t)=e^{-s\wt{\sigma}(x,t)} w(x,t). 
		\end{equation}
		As previously,    we see:
		\begin{itemize}
			\item [($i$)] $\wt{u}=\frac{\partial \wt{u}}{\partial x_i}=0$ at $t=0$ and $t=T$; 
			\item [($ii$)] Let
			\begin{equation*}
				\begin{split}
					P_1 \wt{u}
					&:=-2s\lambda^2\xi \wt{u}  A \nabla \eta\cdot \nabla \eta+ 2s\lambda\xi   A \nabla \wt{u}\cdot \nabla \eta +\wt{u}_t=:(P_1\wt{u})_1+(P_1\wt{u})_2+(P_1\wt{u})_3,\\
					P_2 \wt{u}
					&:=s^2\lambda^2\xi^2 \wt{u}  A \nabla \eta\cdot \nabla \eta+\Div( A \nabla \wt{u})+s \wt{\sigma}_t \wt{u}=:(P_2\wt{u})_1+(P_2\wt{u})_2+(P_2\wt{u})_3,\\
					\wt{f}_{s,\lambda}
					&:=e^{-s \wt{\sigma}} f - s\lambda^2\xi \wt{u}  A \nabla \eta\cdot \nabla \eta - s\lambda \xi \wt{u}\Div( A \nabla \eta) .
				\end{split}
			\end{equation*}
			Then 
			\begin{equation*}
				P_1 \wt{u}+P_2 \wt{u}=\wt{f}_{s,\lambda}.
			\end{equation*}
		\end{itemize}
		Similar to the calculations in   Dirichlet boundary condition case, we have 
		\begin{equation*}
			\begin{split}
				\left( P_1 u, P_2 u \right)_{L^2(Q)}
				\ge&Cs^3\lambda^4\iint_{Q} \xi^3 \left|  A \nabla  \eta \cdot \nabla \eta\right| ^2 u^2 dxdt
				+Cs \lambda^2 \iint_Q \xi  \left|  A \nabla  u \cdot \nabla \eta\right| ^2  dx  dt\\
				&\qquad +C s\lambda^2 \iint_Q \xi   A \nabla \eta \cdot \nabla \eta  A \nabla u \cdot \nabla u dx  dt
				-Cs^3\lambda^4\int_{0}^{T}\int_\om \xi^3 u^2 dxdt\\
				&\qquad  -Cs \lambda \int_{0}^{T}\int_\om \xi   A\nabla u  \cdot \nabla u dx dt
				+B_1 +B_2 +B_3 +B_4 +B_5 +B_6 			
			\end{split}
		\end{equation*}
		and 
		%		By similar calculations and estimates as in the case of Dirichlet boundary condition, we can obtain
		\begin{equation*}
			\begin{split}
				\left( P_1 \wt{u}, P_2 \wt{u} \right)_{L^2(Q)}
				\ge&Cs^3\lambda^4\iint_{Q} \wt{\xi}^3 \left|  A \nabla  \eta \cdot \nabla \eta\right| ^2 \wt{u}^2 dxdt
				+Cs \lambda^2 \iint_Q \wt{\xi}  \left|  A \nabla  \wt{u} \cdot \nabla \eta\right| ^2  dx  dt\\
				&\qquad +C s\lambda^2 \iint_Q \wt{\xi}   A \nabla \eta \cdot \nabla \eta  A \nabla \wt{u} \cdot \nabla \wt{u} dx  dt
				-Cs^3\lambda^4\int_{0}^{T}\int_\om \wt{\xi}^3 \wt{u}^2 dxdt\\
				&\qquad  -Cs \lambda \int_{0}^{T}\int_\om \wt{\xi}   A\nabla \wt{u}  \cdot \nabla \wt{u} dx dt
				+\wt{B}_1 +\wt{B}_2 +\wt{B}_3 +\wt{B}_4 +\wt{B}_5 +\wt{B}_6.	
			\end{split}
		\end{equation*}
		Here, $\wt{B}_1\cdots \wt{B}_6$ denote the boundary terms in the calculations of $\left( P_1 \wt{u}, P_2 \wt{u} \right)_{L^2(Q)}$,
		\begin{equation*}
			\begin{split}
				\wt{B}_1
				=&s^3 \lambda^3 \int_{0}^{T}\int_{\partial\hat{\Om}} \wt{\xi}^3 \wt{u}^2 (A \nabla  \eta \cdot \nabla \eta)   (A  \nabla \eta \cdot \nu)  ds  dt\\
				\wt{B}_2
				=&-2s \lambda^2 \int_{0}^{T} \int_{\partial\hat{\Om}} \wt{\xi} \wt{u}  (A \nabla \eta \cdot \nabla \eta)  ( A \nabla \wt{u} \cdot \nu) ds  dt\\
				\wt{B}_3
				=&2s \lambda \int_{0}^{T}\int_{\partial\hat{\Om}} \wt{\xi} \left(  A\nabla \wt{u}\cdot \nabla \eta \right) (A \nabla \wt{u} \cdot \nu) ds dt\\
				\wt{B}_4
				=&-s \lambda \int_{0}^{T}\int_{\partial\hat{\Om}} \wt{\xi} \left(  A\nabla \wt{u}\cdot \nabla \wt{u} \right) (A\nabla \eta \cdot \nu) ds dt\\
				\wt{B}_5
				=&\iint_\Sigma \wt{u}_t  (A\nabla \wt{u} \cdot \nu) ds dt\\
				\wt{B}_6
				=&-s^2\lambda \int_{0}^{T}\int_{\partial\hat{\Om}} \wt{u}^2 \wt{\xi}  \sigma_t  (A \nabla  \eta \cdot \nu) dsdt.						
			\end{split}
		\end{equation*}
		Replacing $u$ by $e^{-s\sigma}w$ and $\wt{u}$ by $e^{-s\wt{\sigma}}w$ in $B_1$ and $\wt{B}_1$, we see 
		\begin{equation*}
			\begin{split}
				B_1
				=&-s^3 \lambda^3 \int_{0}^{T}\int_{\partial\hat{\Om}} \xi^3 u^2 (A \nabla  \eta \cdot \nabla \eta)   (A  \nabla \eta \cdot \nu)  ds  dt\\
				=&-s^3 \lambda^3 \int_{0}^{T}\int_{\partial\hat{\Om}} \xi^3 u^2 (|\nabla \eta|^2 A\nu \cdot \nu) (-|\nabla \eta|A\nu \cdot \nu) ds  dt\\
				=&s^3 \lambda^3 \int_{0}^{T}\int_{\partial\hat{\Om}} \xi^3 e^{-2s\sigma} w^2 |\nabla \eta|^3 (A\nu \cdot \nu)^2 ds  dt\\
				\ge&0.
			\end{split}
		\end{equation*}
		Note that $\eta=0$ on $\partial\hat{\Om}$, we have $u=\wt{u}$, $\xi=\wt{\xi}$ and $\sigma=\wt{\sigma}$ on $\partial \hat{\Om}$, from which we deduce $\wt{B}_1=-B_1$ and $\wt{B}_1 + B_1 =0$. We already know $A\nabla w \cdot \nu=0$ on $\partial \Omega$ and  $\nabla \eta=0$ on $\partial \hat{\Om}- \partial \hat{\Om}\cap \partial \Om$. In addition, we have $\nabla u = s \lambda e^{-s\sigma} \xi w \nabla \eta + e^{-s\sigma} \nabla w$ and $\nabla \wt{u} = -s \lambda e^{-s\sigma} \xi w \nabla \eta + e^{-s\sigma} \nabla w$. Thus,      we have  for $B_2$ and $\wt{B}_2$, 
		\begin{equation*}
			\begin{split}
				B_2
				=&-2s \lambda^2 \int_{0}^{T}\int_{\partial\hat{\Om}} \xi u (A \nabla  \eta \cdot \nabla \eta)   (A  \nabla u \cdot \nu)  ds  dt\\
				=&2s^2 \lambda^3 \int_{0}^{T}\int_{\partial\hat{\Om}} \xi^2 e^{-2s\sigma}w^2 |\nabla \eta|^3 (A\nu\cdot\nu)^2  ds  dt\\
				&\qquad -2s \lambda^2 \int_{0}^{T}\int_{\partial\hat{\Om}} \xi e^{-2s\sigma}w |\nabla \eta|^2 (A\nu\cdot\nu) (A\nabla w\cdot \nu)  ds  dt\\
				=&2s^2 \lambda^3 \int_{0}^{T}\int_{\partial\hat{\Om}} \xi^2 e^{-2s\sigma}w^2 |\nabla \eta|^3 (A\nu\cdot\nu)^2  ds  dt,
			\end{split}
		\end{equation*}
		and
		\begin{equation*}
			\begin{split}
				\wt{B}_2
				=&-2s \lambda^2 \int_{0}^{T}\int_{\partial\hat{\Om}} \xi \wt{u} (A \nabla  \eta \cdot \nabla \eta)   (A  \nabla \wt{u} \cdot \nu)  ds  dt\\
				=&-2s^2 \lambda^3 \int_{0}^{T}\int_{\partial\hat{\Om}} \xi^2 e^{-2s\sigma}w^2 |\nabla \eta|^3 (A\nu\cdot\nu)^2  ds  dt\\
				&\qquad -2s \lambda^2 \int_{0}^{T}\int_{\partial\hat{\Om}} \xi e^{-2s\sigma}w |\nabla \eta|^2 (A\nu\cdot\nu) (A\nabla w\cdot \nu)  ds  dt\\
				=&-2s^2 \lambda^3 \int_{0}^{T}\int_{\partial\hat{\Om}} \xi^2 e^{-2s\sigma}w^2 |\nabla \eta|^3 (A\nu\cdot\nu)^2  ds  dt.
			\end{split}
		\end{equation*}
		Clearly,   $B_2 + \wt{B}_2$=0. Similarly,
		\begin{equation*}
			\begin{split}
				B_3
				=&-2s \lambda \int_{0}^{T}\int_{\partial\hat{\Om}} \xi \left(  A\nabla u\cdot \nabla \eta \right) (A \nabla u \cdot \nu) ds dt\\
				=&2s^3 \lambda^3 \int_{0}^{T}\int_{\partial\hat{\Om}} \xi^3 e^{-2s\sigma}w^2 |\nabla \eta|^3 (A\nu\cdot\nu)^2  ds  dt\\
				&\qquad + 2s \lambda \int_{0}^{T}\int_{\partial\hat{\Om}} \xi e^{-2s\sigma} |\nabla \eta| (A\nabla w \cdot \nu)^2  ds  dt\\
				&\qquad - 4s^2 \lambda^2 \int_{0}^{T}\int_{\partial\hat{\Om}} \xi^2 e^{-2s\sigma} |\nabla \eta|^2 w (A\nabla w \cdot \nu) (A\nu\cdot\nu) ds  dt\\
				=&2s^3 \lambda^3 \int_{0}^{T}\int_{\partial\hat{\Om}} \xi^3 e^{-2s\sigma}w^2 |\nabla \eta|^3 (A\nu\cdot\nu)^2  ds  dt,
			\end{split}
		\end{equation*}
		and
		\begin{equation*}
			\begin{split}
				\wt{B}_3
				=&2s \lambda \int_{0}^{T}\int_{\partial\hat{\Om}} \xi \left(  A\nabla \wt{u}\cdot \nabla \eta \right) (A \nabla \wt{u} \cdot \nu) ds dt\\
				%			=&-2s^2 \lambda^3 \int_{0}^{T}\int_{\partial\hat{\Om}} \xi^3 e^{-2s\sigma}w^2 |\nabla \eta|^3 (A\nu\cdot\nu)^2  ds  dt\\
				%			&\qquad - 2s \lambda \int_{0}^{T}\int_{\partial\hat{\Om}} \xi e^{-2s\sigma} |\nabla \eta| (A\nabla w \cdot \nu)^2  ds  dt\\
				%			&\qquad - 4s^2 \lambda^2 \int_{0}^{T}\int_{\partial\hat{\Om}} \xi^2 e^{-2s\sigma} |\nabla \eta|^2 w (A\nabla w \cdot \nu) (A\nu\cdot\nu) ds  dt\\
				=&-2s^3 \lambda^3 \int_{0}^{T}\int_{\partial\hat{\Om}} \xi^3 e^{-2s\sigma}w^2 |\nabla \eta|^3 (A\nu\cdot\nu)^2  ds  dt.
			\end{split}
		\end{equation*}
		We also have $B_3 + \wt{B}_3$=0. Moreover,
		\begin{equation*}
			\begin{split}
				B_4
				=&s \lambda \int_{0}^{T}\int_{\partial\hat{\Om}} \xi \left(  A\nabla u\cdot \nabla u \right) (A\nabla \eta \cdot \nu) ds dt\\
				=&-s^3 \lambda^3 \int_{0}^{T}\int_{\partial\hat{\Om}} \xi^3 e^{-2s\sigma} w^2 |\nabla \eta|^3 (A\nu \cdot \nu)^2 ds dt\\
				&\qquad -s \lambda \int_{0}^{T}\int_{\partial\hat{\Om}} \xi e^{-2s\sigma}  |\nabla \eta| (A\nabla w \cdot \nabla w) (A\nu \cdot \nu) ds dt\\
				&\qquad +2s^2 \lambda^2 \int_{0}^{T}\int_{\partial\hat{\Om}} \xi^2 e^{-2s\sigma} w |\nabla \eta|^2 (A\nabla w \cdot \nu) (A\nu \cdot \nu) ds dt\\
				=&-s^3 \lambda^3 \int_{0}^{T}\int_{\partial\hat{\Om}} \xi^3 e^{-2s\sigma} w^2 |\nabla \eta|^3 (A\nu \cdot \nu)^2 ds dt\\
				&\qquad -s \lambda \int_{0}^{T}\int_{\partial\hat{\Om}} \xi e^{-2s\sigma}  |\nabla \eta| (A\nabla w \cdot \nabla w) (A\nu \cdot \nu) ds dt,
			\end{split}
		\end{equation*}
		and
		\begin{equation*}
			\begin{split}
				\wt{B}_4
				=&-s \lambda \int_{0}^{T}\int_{\partial\hat{\Om}} \xi \left(  A\nabla u\cdot \nabla u \right) (A\nabla \eta \cdot \nu) ds dt\\
				=&+s^3 \lambda^3 \int_{0}^{T}\int_{\partial\hat{\Om}} \xi^3 e^{-2s\sigma} w^2 |\nabla \eta|^3 (A\nu \cdot \nu)^2 ds dt\\
				&\qquad +s \lambda \int_{0}^{T}\int_{\partial\hat{\Om}} \xi e^{-2s\sigma}  |\nabla \eta| (A\nabla w \cdot \nabla w) (A\nu \cdot \nu) ds dt.
			\end{split}
		\end{equation*}
		Consequently, $B_4 +\wt{B}_4=0$. Finally, it easy to deduce
		\begin{equation*}
			\begin{split}
				B_5 + \wt{B}_5 =0.
			\end{split}
		\end{equation*}
		and
		\begin{equation*}
			\begin{split}
				B_6 + \wt{B}_6 =0.
			\end{split}
		\end{equation*}
		%	Finally, it easy to deduce
		%	\begin{equation*}
			%		\begin{split}
				%			B_6
				%			=&-s^2\lambda \int_{0}^{T}\int_{\partial\hat{\Om}} u^2 \xi  \sigma_t  (A \nabla  \eta \cdot \nu) dsdt\\
				%			=&s^2\lambda \int_{0}^{T}\int_{\partial\hat{\Om}} e^{-2s\sigma} w^2  \xi  \sigma_t |\nabla \eta| (A \nu \cdot \nu) dsdt.
				%		\end{split}
			%	\end{equation*}
		Summarizing the above  we have 
		$$
		B_1 +\cdots +B_6 +\wt{B}_1 +\cdots + \wt{B}_6 = 0.
		$$
		Hence,
		\begin{equation*}
			\begin{split}
				&\left( P_1 u,  P_2 u \right)_{L^2(Q)} + \left( P_1 \wt{u}, P_2 \wt{u} \right)_{L^2(Q)}\\
				&\qquad\ge Cs^3\lambda^4\iint_{Q} \xi^3 \left|  A \nabla  \eta \cdot \nabla \eta\right| ^2 u^2 dxdt
				+Cs \lambda^2 \iint_Q \xi  \left|  A \nabla  u \cdot \nabla \eta\right| ^2  dx  dt\\
				&\qquad\qquad +C s\lambda^2 \iint_Q \xi   A \nabla \eta \cdot \nabla \eta  A \nabla u \cdot \nabla u dx  dt
				-Cs^3\lambda^4\int_{0}^{T}\int_\om \xi^3 u^2 dxdt\\
				&\qquad\qquad  -Cs \lambda \int_{0}^{T}\int_\om \xi   A\nabla u  \cdot \nabla u dx dt\\
				&\qquad\qquad +Cs^3\lambda^4\iint_{Q} \wt{\xi}^3 \left|  A \nabla  \eta \cdot \nabla \eta\right| ^2 \wt{u}^2 dxdt
				+Cs \lambda^2 \iint_Q \wt{\xi}  \left|  A \nabla  \wt{u} \cdot \nabla \eta\right| ^2  dx  dt\\
				&\qquad\qquad +C s\lambda^2 \iint_Q \wt{\xi}   A \nabla \eta \cdot \nabla \eta  A \nabla \wt{u} \cdot \nabla \wt{u} dx  dt
				-Cs^3\lambda^4\int_{0}^{T}\int_\om \wt{\xi}^3 \wt{u}^2 dxdt\\
				&\qquad\qquad  -Cs \lambda \int_{0}^{T}\int_\om \wt{\xi}   A\nabla \wt{u}  \cdot \nabla \wt{u} dx dt.			
			\end{split}
		\end{equation*}
		We also have
		\begin{equation*}
			\begin{split}
				\left\| P_1 u \right\| ^2_{L^2(Q)} +& \left\| P_2 u \right\| ^2_{L^2(Q)} +\left\| P_1 \wt{u} \right\| ^2_{L^2(Q)} + \left\| P_2 \wt{u} \right\| ^2_{L^2(Q)}\\%1
				&\quad +Cs^3\lambda^4\iint_Q \xi^3 \left|  A \nabla  \eta \cdot \nabla \eta\right| ^2 u^2 dxdt
				+Cs \lambda^2 \iint_Q \xi  \left|  A \nabla  u \cdot \nabla \eta\right| ^2  dx  dt\\%2
				&\quad+C s\lambda^2 \iint_Q \xi   A \nabla \eta \cdot \nabla \eta  A \nabla u \cdot \nabla u dx  dt\\%3
				&\quad +Cs^3\lambda^4\iint_Q \wt{\xi}^3 \left|  A \nabla  \eta \cdot \nabla \eta\right| ^2 \wt{u}^2 dxdt
				+Cs \lambda^2 \iint_Q \wt{\xi}  \left|  A \nabla  \wt{u} \cdot \nabla \eta\right| ^2  dx  dt\\%4
				&\quad+C s\lambda^2 \iint_Q \wt{\xi}   A \nabla \eta \cdot \nabla \eta  A \nabla \wt{u} \cdot \nabla \wt{u} dx  dt\\%5
				\le& \left\| e^{-s\sigma} f \right\| ^2_{L^2(Q)}
				+ Cs^3\lambda^4\int_{0}^{T}\int_\om \xi^3 u^2 dxdt	
				+Cs \lambda \int_{0}^{T}\int_\om \xi   A\nabla u  \cdot \nabla u dx dt\\
				&\quad+\left\| e^{-s\wt{\sigma}} f \right\| ^2_{L^2(Q)}+ Cs^3\lambda^4\int_{0}^{T}\int_\om \wt{\xi}^3 \wt{u}^2 dxdt	
				+Cs \lambda \int_{0}^{T}\int_\om \wt{\xi}   A\nabla \wt{u}  \cdot \nabla \wt{u} dx dt.
			\end{split}
		\end{equation*}
		With  a similar  calculation as before, we   obtain
		\begin{equation}\label{PP4}
			\begin{split}
				Cs^{-1} \iint_Q & \xi^{-1}\left|u_t\right|^2dx  d t +Cs^{-1}   \iint_Q \xi^{-1}\left|\Div( A\nabla u)\right|^2dx  d t\\
				&\qquad+Cs^3\lambda^4\iint_Q \xi^3  u^2 dxdt
				+Cs \lambda^2 \iint_Q \xi  \left|  A \nabla  u \cdot \nabla \eta\right| ^2  dx  dt\\
				&\qquad+C s\lambda^2 \iint_Q \xi    A \nabla u \cdot \nabla u dx  dt\\
				&\qquad+Cs^{-1} \iint_Q  \wt{\xi}^{-1}\left|\wt{u}_t\right|^2dx  d t +Cs^{-1}   \iint_Q \wt{\xi}^{-1}\left|\Div( A\nabla \wt{u})\right|^2dx  d t\\
				&\qquad+Cs^3\lambda^4\iint_Q \wt{\xi}^3  \wt{u}^2 dxdt
				+Cs \lambda^2 \iint_Q \wt{\xi}  \left|  A \nabla  \wt{u} \cdot \nabla \eta\right| ^2  dx  dt\\
				&\qquad+C s\lambda^2 \iint_Q \wt{\xi}    A \nabla \wt{u} \cdot \nabla \wt{u} dx  dt\\
				\le& \left\| e^{-s\sigma} f \right\| ^2_{L^2(Q)}
				+ Cs^3\lambda^4\int_{0}^{T}\int_{\om_0} \xi^3 u^2 dxdt\\
				&\qquad+\left\| e^{-s\wt{\sigma}} f \right\| ^2_{L^2(Q)}
				+ Cs^3\lambda^4\int_{0}^{T}\int_{\om_0} \wt{\xi}^3 \wt{u}^2 dxdt.
			\end{split}
		\end{equation}
		Replacing $u$ by $e^{-s\sigma}w$ and $\wt{u}$ by $e^{-s\wt{\sigma}}w$ in \eqref{PP4}, we get
		\begin{equation*}
			\begin{split}
				Cs^{-1} \iint_Q & e^{-2s\sigma}\xi^{-1}\left|w_t\right|^2dx  d t +Cs^{-1}   \iint_Q e^{-2s\sigma}\xi^{-1}\left|\Div( A\nabla w)\right|^2dx  d t\\%1
				&\qquad+Cs^3\lambda^4\iint_Q e^{-2s\sigma}\xi^3  w^2 dxdt
				+Cs \lambda^2 \iint_Q e^{-2s\sigma}\xi  \left|  A \nabla  w \cdot \nabla \eta\right| ^2  dx  dt\\%2
				&\qquad+C s\lambda^2 \iint_Q e^{-2s\sigma}\xi    A \nabla w \cdot \nabla w dx  dt
				+Cs^{-1} \iint_Q  e^{-2s\wt{\sigma}}\wt{\xi}^{-1}\left|w_t\right|^2dx  d t\\%3
				&\qquad+Cs^{-1}   \iint_Q e^{-2s\wt{\sigma}}\wt{\xi}^{-1}\left|\Div( A\nabla w)\right|^2dx  d t
				+Cs^3\lambda^4\iint_Q e^{-2s\wt{\sigma}}\wt{\xi}^3  w^2 dxdt\\
				&\qquad+Cs \lambda^2 \iint_Q e^{-2s\wt{\sigma}}\wt{\xi}  \left|  A \nabla  w \cdot \nabla \eta\right| ^2  dx  dt
				C s\lambda^2 \iint_Q e^{-2s\wt{\sigma}}\wt{\xi}    A \nabla w \cdot \nabla w dx  dt\\
				\le& \left\| e^{-s\sigma} f \right\| ^2_{L^2(Q)}
				+ Cs^3\lambda^4\int_{0}^{T}\int_{\om_0} e^{-2s\sigma}\xi^3 w^2 dxdt\\
				&\qquad+\left\| e^{-s\wt{\sigma}} f \right\| ^2_{L^2(Q)}
				+ Cs^3\lambda^4\int_{0}^{T}\int_{\om_0} e^{-2s\wt{\sigma}}\wt{\wt{\xi}}^3 w^2 dxdt.
			\end{split}
		\end{equation*}
		Since $\xi$ and $\wt{\xi}$ are bounded functions on $\Omega$, there exist different constants $C_1, C_2$ and $C_3$ such that
		\begin{equation*}
			C_1 |\xi| \le |\wt{\xi}| \le C_2 |\xi|, \ C_3 \frac{1}{|\xi|} \le \frac{1}{|\wt{\xi}|} \le C_2 \frac{1}{|\xi|}, \ \forall (x,t)\in Q.
		\end{equation*}
		Then we have
		\begin{equation*}
			\begin{split}
				C\iint_Q&  \left(  s^{-1} \xi^{-1} \left( |w_t|^2 + |\Div(A\nabla w)|^2\right)  + s\lambda^2 \xi|A\nabla w \cdot \nabla \eta|^2 + s\lambda^2 \xi A\nabla w \cdot \nabla w\right. \\
				&\left. \qquad+ s^3 \lambda^4 \xi^3 w^2 dx \right) \left(e^{2s\sigma} + e^{2s\wt{\sigma}} \right)    d t\\%3
				\le& C \iint_Q |f|^2\left(e^{2s\sigma} + e^{2s\wt{\sigma}} \right) dx dt 
				+ C\int_{0}^{T}\int_{\om_0} s^3\lambda^4\xi^3 w^2 \left(e^{2s\sigma} + e^{2s\wt{\sigma}} \right) dxdt.
			\end{split}
		\end{equation*}
		This completes the proof of Carleman inequality for Neumann  boundary condition. 
	\end{proof}
	
	The way that the Carleman estimate in Theorem \ref{TH2} leads to Theorem \ref{TH4} is standard. 
	
	\begin{proof}[{\bf Proof of Theorem \ref{TH4}}]	
		Below, we   prove the observability inequality only for   Neumann boundary condition, while a similar proof can be applied for  Dirichlet boundary condition. Firstly, multiply the equation in \eqref{3.1} by $w$, and  letting  $f=0$ yield   that, for all $t \ge 0$,
		\begin{equation*}
			\frac{d}{d t} \int_{\Omega} w(x, t)^2 d x=2 \iint_{Q} A(x) \nabla w \cdot \nabla w d x d t \ge 0,
		\end{equation*}
		Hence, the function
		\begin{equation*}
			t \mapsto \int_{\Omega} w(x, t)^2 d x
		\end{equation*}
		is nondecreasing and it follows that
		\begin{equation*}
			\int_{\Omega} w(x, 0)^2 d x \leq \frac{2}{T} \int_{T / 4}^{3 T / 4} \int_{\Omega} w(x, t)^2 d x d t.
		\end{equation*}
		Next, choosing  $s_0, \lambda_0$ and $\xi$   as in Theorem \ref{TH2}   we have
		\begin{equation*}
			\begin{split}
				& \int_{\Omega} w(x, 0)^2 d x \leq \frac{2}{T} \int_{T / 4}^{3 T / 4} \int_{\Omega} w(x, t)^2 d x d t \\%1
				& \leq \frac{2}{T} \frac{1}{\inf _{\left(\frac{T}{4}, \frac{3 T}{4}\right) \times \Omega} s^2 \lambda^2  \xi^2 \left(e^{2s\sigma} + e^{2s\wt{\sigma}} \right)  } \int_{T / 4}^{3 T / 4} \int_{\Omega} s^2 \lambda^2  \xi^2  \left(e^{2s\sigma} + e^{2s\wt{\sigma}} \right)   w(x, t)^2 d x d t \\%2
				& \leq \frac{2}{T} \frac{1}{\inf _{\left(\frac{T}{4}, \frac{3 T}{4}\right) \times \Omega} s^2 \lambda^2  \xi^2 \left(e^{2s\sigma} + e^{2s\wt{\sigma}} \right) } \int_{0}^{T} \int_{\Omega} s^2 \lambda^2  \xi^2  \left(e^{2s\sigma} + e^{2s\wt{\sigma}} \right)   w(x, t)^2 d x d t. 
			\end{split}
		\end{equation*}
		Using Theorem \ref{TH2},   there exists $C=C(\Omega, \omega_{0}, T)$ such that, for some fixed $\lambda \geq \lambda_0$ and $s \geq s_0$, we have
		$$
		\int_{0}^{T} \int_{\Omega} s^2 \lambda^2  \xi^2  \left(e^{2s\sigma} + e^{2s\wt{\sigma}} \right)   w(x, t)^2 d x d t 
		\le C s^3 \lambda^3 \int_0^T \int_{\omega_{0}} \left(e^{2s\sigma} + e^{2s\wt{\sigma}} \right)  \xi^3w^2 dx dt,
		$$
		and hence
		$$
		\begin{aligned}
			& \int_{\Omega} w(x, 0)^2 d x \\
			&\leq 
			\frac{C s^3 \lambda^3}{T \inf _{\left(\frac{T}{4}, \frac{3 T}{4}\right) \times \Omega} s^2 \lambda^2  \xi^2 \left(e^{2s\sigma} + e^{2s\wt{\sigma}} \right)}      \int_0^T \int_{\omega_{0}} \left(e^{2s\sigma} + e^{2s\wt{\sigma}} \right)  \xi^3 w^2 dx dt  \\
			&  \leq C(\Omega, \omega_{0}, T) \int_0^T \int_{\omega_{0}} w^2. \\
			&
		\end{aligned}
		$$
		This proves Theorem \ref{TH4} for Neumann boundary condition. 	Similarly, we can prove the theorem for  Dirichlet boundary condition.
	\end{proof}

	\vspace{3mm}
	
	\noindent{\bf Acknowledgement}
	
	\vspace{2mm}
	
	This work is supported by the National Natural Science Foundation of China, the Natural Sciences and Engineering Research Council of Canada (RGPIN-2018-05687 and RGPIN 2024-05941), and a centennial fund of the University of Alberta.

	% 	\bibliographystyle{abbrvnat}
	
	% \bibliographystyle{accountm}
	\bibliographystyle{abbrv}

	\bibliography{ref20230701}

\begin{thebibliography}{10}

\bibitem{Allal2021null}
B.~Allal and G.~Fragnelli.
\newblock Null controllability of degenerate parabolic equation with memory.
\newblock {\em Mathematical Methods in the Applied Sciences}, 44(11):9163--9190, 2021.

\bibitem{allal2020lipschitz}
B.~Allal, A.~Hajjaj, L.~Maniar, and J.~Salhi.
\newblock Lipschitz stability for some coupled degenerate parabolic systems with locally distributed observations of one component.
\newblock {\em Mathematical Control \& Related Fields}, 10(3):643, 2020.

\bibitem{anh2013null}
C.~T. Anh and V.~M. Toi.
\newblock Null controllability of a parabolic equation involving the grushin operator in some multi-dimensional domains.
\newblock {\em Nonlinear Analysis: Theory, Methods and Applications}, 93:181--196, 2013.

\bibitem{BSV}
F.~D. Araruna, B.~S.~V. Arajuo, and E.~Fernandez-Cara.
\newblock Carleman estimates for some two-dimensional degenerate parabolic pdes and applications.
\newblock {\em SIAM Journal on Control and Optimization}, 57(6):3985--4010, 2019.

\bibitem{araujo2022boundary}
B.~S.~V. Ara{\'u}jo, R.~Demarque, and L.~Viana.
\newblock Boundary null controllability of degenerate heat equation as the limit of internal controllability.
\newblock {\em Nonlinear Analysis: Real World Applications}, 66:103519, 2022.

\bibitem{banerjee2022carleman}
A.~Banerjee, N.~Garofalo, and R.~Manna.
\newblock Carleman estimates for baouendi--grushin operators with applications to quantitative uniqueness and strong unique continuation.
\newblock {\em Applicable Analysis}, 101(10):3667--3688, 2022.

\bibitem{beauchard20152d}
K.~Beauchard, L.~Miller, and M.~Morancey.
\newblock 2d grushin-type equations: minimal time and null controllable data.
\newblock {\em Journal of Differential Equations}, 259(11):5813--5845, 2015.

\bibitem{benoit2023null}
A.~Benoit, R.~Loyer, and L.~Rosier.
\newblock Null controllability of strongly degenerate parabolic equations.
\newblock {\em ESAIM: Control, Optimisation and Calculus of Variations}, 29:48, 2023.

\bibitem{bensoussan2007representation}
A.~Bensoussan, G.~Da~Prato, M.~C. Delfour, and S.~K. Mitter.
\newblock {\em Representation and control of infinite dimensional systems}, volume~2.
\newblock Springer, 2007.

\bibitem{boutaayamou2018carleman}
I.~Boutaayamou, G.~Fragnelli, and L.~Maniar.
\newblock Carleman estimates for parabolic equations with interior degeneracy and neumann boundary conditions.
\newblock {\em Journal d'Analyse Math{\'e}matique}, 135:1--35, 2018.

\bibitem{boutaayamou2016null}
I.~Boutaayamou and J.~Salhi.
\newblock Null controllability for linear parabolic cascade systems with interior degeneracy.
\newblock {\em Electron. J. Differential Equations}, 2016(305):1--22, 2016.

\bibitem{calin2009heat}
O.~Calin, D.-C. Chang, and H.~Fan.
\newblock The heat kernel for kolmogorov type operators and its applications.
\newblock {\em Journal of Fourier Analysis and Applications}, 15(6):816, 2009.

\bibitem{calin2010heat}
O.~Calin, D.-C. Chang, K.~Furutani, and C.~Iwasaki.
\newblock {\em Heat kernels for elliptic and sub-elliptic operators}.
\newblock Springer, 2010.

\bibitem{cannarsa2004persistent}
P.~Cannarsa, P.~Martinez, and J.~Vancostenoble.
\newblock Persistent regional null controllability for a class of degenerate parabolic equations.
\newblock {\em Commun. Pure Appl. Anal}, 3(4):607--635, 2004.

\bibitem{cannarsa2016global}
P.~Cannarsa, P.~Martinez, and J.~Vancostenoble.
\newblock {\em Global Carleman estimates for degenerate parabolic operators with applications}, volume 239.
\newblock American Mathematical Society, 2016.

\bibitem{CA5}
P.~Cannarsa, D.~Rocchetti, J.~Vancostenoble, et~al.
\newblock Generation of analytic semi-groups in l-2 for a class of second order degenerate elliptic operators.
\newblock {\em Control and Cybernetics}, 37:831--878, 2008.

\bibitem{cazenave1998introduction}
T.~Cazenave and A.~Haraux.
\newblock {\em An introduction to semilinear evolution equations}, volume~13.
\newblock Oxford University Press on Demand, 1998.

\bibitem{chen2018time}
N.~Chen, Y.~Wang, and D.-H. Yang.
\newblock Time-varying bang--bang property of time optimal controls for heat equation and its application.
\newblock {\em Systems \& Control Letters}, 112:18--23, 2018.

\bibitem{de2023null}
P.~de~Carvalho, R.~Demarque, J.~L{\'\i}maco, and L.~Viana.
\newblock Null controllability and numerical simulations for a class of degenerate parabolic equations with nonlocal nonlinearities.
\newblock {\em Nonlinear Differential Equations and Applications NoDEA}, 30(3):32, 2023.

\bibitem{flores2010carleman}
C.~Flores and L.~de~Teresa.
\newblock Carleman estimates for degenerate parabolic equations with first order terms and applications.
\newblock {\em Comptes Rendus Mathematique}, 348(7-8):391--396, 2010.

\bibitem{flores2020null}
J.~C. Flores and L.~D. Teresa.
\newblock Null controllability of one dimensional degenerate parabolic equations with first order terms.
\newblock {\em Discrete \& Continuous Dynamical Systems-Series B}, 25(10), 2020.

\bibitem{lin2007some}
P.~Lin, H.~Gao, and X.~Liu.
\newblock Some results on controllability of a nonlinear degenerate parabolic system by bilinear control.
\newblock {\em Journal of mathematical analysis and applications}, 326(2):1149--1160, 2007.

\bibitem{lions1992remarks}
J.-L. Lions.
\newblock Remarks on approximate controllability.
\newblock {\em Journal d’Analyse Math{\'e}matique}, 59(1):103--116, 1992.

\bibitem{liu2019carleman}
X.~Liu and Y.~Yu.
\newblock Carleman estimates of some stochastic degenerate parabolic equations and application.
\newblock {\em SIAM Journal on Control and Optimization}, 57(5):3527--3552, 2019.

\bibitem{liu2023some}
Y.~Liu, H.~Sun, W.~Wu, and D.~Yang.
\newblock Carleman estimates for degenerate parabolic equations with single interior point degeneracy and its applications.
\newblock {\em arXiv preprint arXiv:2308.10012}, 2023.

\bibitem{liu2024observability}
Y.~Liu, W.~Wu, D.~Yang, and C.~Zhang.
\newblock Observability inequality from measurable sets for degenerate parabolic equations and its applications.
\newblock {\em Journal of Optimization Theory and Applications}, pages 1--39, 2024.

\bibitem{marinoschi2017optimal}
G.~Marinoschi, R.~M. Mininni, and S.~Romanelli.
\newblock An optimal control problem in coefficients for a strongly degenerate parabolic equation with interior degeneracy.
\newblock {\em Journal of Optimization Theory and Applications}, 173:56--77, 2017.

\bibitem{martinez2003regional}
P.~Martinez, J.-P. Raymond, and J.~Vancostenoble.
\newblock Regional null controllability of a linearized {C}rocco-type equation.
\newblock {\em SIAM J. Control Optim.}, 42(2):709--728, 2003.

\bibitem{rockafellar1967duality}
R.~Rockafellar.
\newblock Duality and stability in extremum problems involving convex functions.
\newblock {\em Pacific Journal of Mathematics}, 21(1):167--187, 1967.

\bibitem{sakthivel2008exact}
K.~Sakthivel, K.~Balachandran, R.~Sowrirajan, and J.-H. Kim.
\newblock On exact null controllability of black-scholes equation.
\newblock {\em Kybernetika}, 44(5):685--704, 2008.

\bibitem{showalter1995hilbert}
R.~Showalter.
\newblock Hilbert space methods for partial differential equations.
\newblock 1995.

\bibitem{wang2023carleman}
C.~Wang and Y.~Zhou.
\newblock Carleman estimate and null controllability for a degenerate parabolic equation with a slightly superlinear reaction term.
\newblock {\em Nonlinear Differential Equations and Applications NoDEA}, 30(5):69, 2023.

\bibitem{wu2020carleman}
B.~Wu, Q.~Chen, and Z.~Wang.
\newblock Carleman estimates for a stochastic degenerate parabolic equation and applications to null controllability and an inverse random source problem.
\newblock {\em Inverse Problems}, 36(7):075014, 2020.

\bibitem{zheng2024global}
Y.~Zheng and Y.~Hu.
\newblock The global maximum principle for optimal control of partially observed stochastic systems driven by fractional brownian motion.
\newblock {\em SIAM Journal on Control and Optimization}, 62(1):509--538, 2024.

\end{thebibliography}
	%\begin{thebibliography}{10}
	%	
	%	\bibitem{Evans} L.C. Evans, Partial Differential Equations, American Mathematical Society Providence, Rholde Island, 2010. 
	%	
	%	
	%\end{thebibliography}
\end{document}